\documentclass[11pt,a4paper]{article}
\usepackage{epsf,epsfig,amsfonts,amsgen,amsmath,amssymb,amstext,amsbsy,amsopn,amsthm,lineno}
\usepackage{color} 
\usepackage{graphicx}
\usepackage{caption}
\usepackage{lineno}
\usepackage{amsthm}
\usepackage{thmtools}
\usepackage{amsmath}
\usepackage{amsfonts}
\usepackage{amssymb}
\usepackage{yfonts}
\usepackage{centernot}
\usepackage{epsfig}
\usepackage{graphicx}
\usepackage{mathabx}
\usepackage{enumitem}
\usepackage{float}
\usepackage{tkz-graph}  
\usepackage{csquotes}  
\usepackage{crop}

\usetikzlibrary{calc,arrows.meta,positioning}
\usetikzlibrary{decorations.markings}

\GraphInit[vstyle = Normal] 
\tikzset{
  LabelStyle/.style = {font = \tiny\bfseries },
  VertexStyle/.append style = { inner sep=5pt,
                                font = \tiny\bfseries},
  EdgeStyle/.append style = {->} }
\usetikzlibrary{arrows, shapes, positioning}

\newcommand{\overundercup}[2]{{\underset{#1}{\overset{#2}{\cup}}}}

\newcommand{\oversyncprod}[1]{{\overset{#1}{\boxbackslash}}}

\newcommand{\overundersyncprod}[2]{{\underset{ #1}{\oversyncprod #2}}}

\makeatletter

\newtheorem{theorem}{Theorem}

\theoremstyle{definition}

\newtheorem{claim}{Claim}

\begin{document}
\title
{\bf\Large Three new decompositions of graphs \\based on a vertex-removing synchronised graph product} 
\date{}
\author{\small  Antoon H. Boode \\[2mm]
\small Faculty of EEMCS, University of Twente, P.O. Box 217,\\
\small 7500 AE Enschede, The Netherlands}
\maketitle
\begin{abstract}
\noindent
Recently, we have introduced and modified two graph-decomposition theorems based on a new graph product, motivated by applications in the context of synchronising periodic real-time processes. 
This vertex-removing synchronised product (VRSP), is based on modifications of the well-known Cartesian product, and is closely related to the synchronised product due to W\"ohrle and Thomas. 
Here, we recall the definition of the VRSP and the two modified graph-decompositions and introduce three new graph-decomposition theorems. 
The first new theorem decomposes a graph with respect to the semicomplete bipartite subgraphs of the graph.
For the second new theorem, we introduce a matrix graph, which is used to decompose a graph in a manner similar to the decomposition of graphs using the Cartesian product.
In the third new theorem, we combine these two types of decomposition.
Ultimately, the goal of these graph-decomposition theorems is to come to a prime-graph decomposition.
\medskip



\end{abstract}

\section{Introduction}\label{sec:intro}
Recently, we have introduced~\cite{dam} and modified~\cite{ejgta1} two graph-decomposition theorems based on a new graph product, motivated by applications in the context of synchronising periodic real-time processes, in particular in the field of robotics. 
More on the background, definitions and applications can be found in two conference contributions \cite{boode2014cpa, boode2013cpa}, two journal papers \cite{dam,ejgta} and the thesis of the 
author~\cite{boodethesis}. 
We repeat some of the background, definitions and theorems here for convenience, and for supplying the motivation for the research that led to the third, fourth and fifth decomposition theorem 
that we state and prove in Section~\ref{sec:decomp}. 

The decomposition of graphs is well known in the literature.
For example, a decomposition can be based on the partition of a graph into edge disjoint subgraphs. 
In our case, the decomposition is based on the contraction of a subset of the vertices of the graph, in such a manner that if $V'\subset V(G)$ is contracted giving $G'$ and $V''\subset V(G)$ is contracted giving $G''$ we have that the vertex-removing synchronised product (VRSP) of $G'$ and $G''$ is isomorphic to $G$.

The rest of the paper is organised as follows.
In the next sections, we first recall the formal graph definitions (in Section~\ref{sec:term}), the 
definition of the VRSP as well as the graph-decomposition theorems, together with other relevant terminology and notation (in Section~\ref{Terminology_products}), and the notions of graph isomorphism and contraction to labelled acyclic directed multigraphs (in Section~\ref{Terminology_morphisms}).
%
Finally, we prove (in Section~\ref{sec:decomp}) the third, fourth and fifth decomposition theorem.

\section{Terminology and notation}\label{sec:term}
We use the textbook of 
Bondy and Murty \cite{GraphTheory} for terminology and notation we do not specify here. 
Throughout, unless we specify explicitly that we consider other types of graphs, all graphs we consider are {\em labelled acyclic directed multigraphs\/}, i.e., they may have multiple arcs. Such graphs consist of a {\em vertex set\/} $V$ (representing the states of a process), an {\em arc set\/}  $A$ (representing the actions, i.e., transitions from one state to another), a set of {\em labels\/} $L$ (in our applications in fact a set of label pairs, each representing a type of action and the worst case duration of its execution), and two mappings.
The first mapping $\mu: A\rightarrow V\times V$ is an incidence function that identifies the {\em tail\/} and {\em head\/} of each arc $a\in A$. 
In particular, $\mu(a)=(u,v)$ means that the arc $a$ is directed from $u\in V$ to $v\in V$, where $tail(a)=u$ and $head(a)=v$. We also call $u$ and $v$ the {\em ends\/} of $a$. 
The second mapping $\lambda :A\rightarrow L$ assigns a label pair $\lambda(a)=(\ell(a),t(a))$ to each arc $a\in A$, where $\ell(a)$ is a string representing the (name of an) action and $t(a)$ is the {\em weight\/} of the arc $a$.
This weight $t(a)$ is a real positive number representing the worst case execution time of the action represented by $\ell(a)$.

Let $G$ denote a graph according to the above definition.
An arc $a\in A(G)$ is called an {\em in-arc\/} of $v\in V(G)$ if $head(a)=v$, and an {\em out-arc\/} of $v$ if $tail(a)=v$. The {\em in-degree\/} of $v$, denoted by $d^-(v)$, is the number of in-arcs of $v$ in $G$; the {\em out-degree\/} of $v$, denoted by $d^+(v)$, is the number of out-arcs of $v$ in $G$.
The subset of $V(G)$ consisting of vertices $v$ with $d^-(v)=0$ is called the {\em source\/} of $G$, and is denoted by $S'(G)$. The subset of $V(G)$ consisting of vertices $v$ with $d^+(v)=0$ is called the {\em sink\/} of $G$, and is denoted by $S''(G)$.

For disjoint nonempty sets $X,Y\subseteq V(G)$, $[X,Y]$ denotes the set of arcs of $G$ with one end in $X$ and one end in $Y$. If the head of the arc $a\in [X,Y]$ is in $Y$, we call $a$ a {\em forward arc\/} (of $[X,Y]$); otherwise, we call it a {\em backward arc\/}. 

The acyclicity of $G$ implies a natural ordering of the vertices into disjoint sets, as follows.
We define $S^0(G)$ to denote the set of vertices with in-degree 0 in $G$ (so $S^0(G)=S'(G)$), $S^1(G)$ the set of vertices with in-degree 0 in the graph obtained from $G$ by deleting the vertices of $S^0(G)$ and all arcs with tails in $S^0(G)$, and so on, until the final set $S^{t}(G)$ contains the remaining vertices with in-degree 0 and out-degree 0 in the remaining graph. Note that these sets are well-defined since $G$ is acyclic, and also note that $S^{t}(G)\neq S''(G)$, in general.  
If a vertex $v\in V(G)$ is in the set $S^j(G)$ in the above ordering, we say that $v$ is {\em at $level\, j$\/} in $G$.
This ordering implies that each arc $a\in A(G)$ can only have $tail(a)\in S^{j_1}(G)$ and  $head(a)\in S^{j_2}(G)$ if $j_1<j_2$.

A graph $G$ is called {\em weakly connected\/} if all pairs of distinct vertices $u$ and $v$ of $G$ are connected through a 
sequence of distinct vertices $u=v_0v_1\ldots  v_k=v$ and arcs $a_1a_2\ldots a_k$ of $G$ with $\mu(a_i) = (v_{i-1}, v_i)$ or $(v_{i},v_{i-1})$ for $i=1,2,\ldots ,k$. 
We are mainly interested in weakly connected graphs, or in the weakly connected components of a graph $G$. 
If $X\subseteq V(G)$, then the {\em subgraph of $G$ induced by $X$\/}, denoted as $G[X]$, is the graph on vertex set $X$ containing all the arcs of $G$ which have both their ends in $X$ (together with $L$, $\mu$ and $\lambda$ restricted to this subset of the arcs). If $X\subseteq V$ induces a weakly connected subgraph of $G$, but there is no set $Y\subseteq V$ such that $G[Y]$ is weakly connected and $X$ is a proper subset of $Y$, then $G[X]$ is called a {\em weakly connected component\/} of $G$. Also, the set of arcs of $G[X]$ is denoted as $A[X]$. 
If $X\subseteq A(G)$, then the {\em subgraph of $G$ arc-induced by $X$\/}, denoted as $G\{X\}$, is the graph on arc set $X$ containing all the vertices of $G$ which are an end of an arc in $X$ (together with $L$, $\mu$ and $\lambda$ restricted to this subset of the arcs). 
If $X\subseteq A$ arc-induces a weakly connected subgraph of $G$, but there is no set $Y\subseteq A$ such that $G\{Y\}$ is weakly connected and $X$ is a proper subset of $Y$, then $G\{X\}$ is called a {\em weakly connected component\/} of $G$.

%
In the sequel, throughout we omit the words weakly connected, so a component should always be understood as a weakly connected component. In contrast to the notation in the textbook of Bondy and Murty \cite{GraphTheory},
we use $\omega(G)$ to denote the number of components of a graph $G$.

We denote the components of $G$ by $G_i$, where $i$ ranges from 1 to $\omega(G)$.
In that case, we use $V_i$, $A_i$ and $L_i$ as shorthand notation for $V(G_i)$, $A(G_i)$ and $L(G_i)$, respectively.
The mappings $\mu$ and $\lambda$ have natural counterparts restricted to the subsets $A_i\subset A(G)$ that we do not specify explicitly. 
We use $G=\sum\limits_{i=1}^{\omega(G)} G_i$ to indicate that $G$ is the disjoint union of its components, implicitly defining its components as $G_1$ up to $G_{\omega(G)}$. In particular, $G=G_1$ if and only if $G$ is weakly connected itself.
Furthermore, we use $\overundercup{i=1}{\omega(G)} G_i$ to denote the graph with vertex set $\overundercup{i=1}{\omega(G)} V_i$, arc set $\overundercup{i=1}{\omega(G)} A_i$ with the mappings $\mu_i(a_i)=(u_i,v_i)$ and $\lambda(a_i)=(\ell(a_i),t(a_i))$ for each arc $a_i\in A_i$.

A subgraph $B$ of $G$ according to the above definition is called \emph{bi-partite} if there exists a partition of non-empty sets $V_1$ and $V_2$ of $V(B)$ into two partite sets (i.e., $V(B) = V_1 \cup V_2$, $V_1 \cap V_2 = \emptyset$) such that every arc of $B$ has its head vertex and tail vertex in different partite sets. Such a graph is called a \emph{bipartite subgraph}, and we denote such a bipartite subgraph of $G$ by $B(V_1, V_2)$. 
A bipartite graph $B(V_1, V_2)$ is called complete if, for every pair $x \in V_1$, $y \in V_2$, there is an arc $a$ met $\mu(a)=(x,y)$ or $\mu(a)=(y,x)$ in $B(V_1, V_2)$.
We call $B(V_1, V_2)$ a trivial bipartite graph if $|V_1|=|V_2|=1$.
A bipartite subgraph $B(V_1,V_2)$ is semicomplete if, for every pair $x \in V_1$, $y \in V_2$, an arc $xy$ is in $B(V_1,V_2)$ or an arc $yx$ is in $B(V_1,V_2)$, or for every pair $x \in V_1$, $y \in V_2$, there is no arc $xy$ in $B(V_1,V_2)$ and there is no arc $yx$ in $B(V_1,V_2)$.

If necessary, we divide $V$ into mutually disjoint subsets with a cardinality that is a prime number. 
We denote the union of mutually disjoint subsets $V_1,\ldots, V_{n}$ of $V$ with the same cardinality $p_i$ as $(V^{p_i})^{n}$. Hence, $|(V^{p_i})^{n}|=n\cdot p_i$.

In the next two sections, we recall some of the definitions that appeared in~\cite{dam}.

\section{Graph products}\label{Terminology_products}
Instead of defining products for general pairs of graphs, 
for notational reasons we find it convenient to define those products for two components $G_i$ and $G_j$ of a disconnected  graph $G$. We start with the next analogue of the Cartesian product.

The {\em Cartesian product\/} $G_i\Box G_j$ of $G_i$ and $G_j$ is defined as the graph on vertex set $V_{i,j}=V_i\times V_j$, and arc set $A_{i,j}$ consisting of two types of labelled arcs. 
For each arc $a\in A_i$ with $\mu(a)=(v_i,w_i)$, an {\em arc of type $i$\/} is introduced between tail $(v_i,v_j)\in V_{i,j}$ and head $(w_i,w_j)\in V_{i,j}$ whenever $v_j=w_j$; such an arc receives the label $\lambda(a)$. 
This implicitly defines parts of the mappings $\mu$ and $\lambda$ for $G_i \Box G_j$.
Similarly, for each arc $a\in A_j$ with $\mu(a)=(v_j,w_j)$, an {\em arc of type $j$\/} is introduced between tail $(v_i,v_j)\in V_{i,j}$ and head $(w_i,w_j)\in V_{i,j}$ whenever $v_i=w_i$; such an arc receives the label $\lambda(a)$. 
This completes the definition of $A_{i,j}$ and the  mappings $\mu$ and $\lambda$ for $G_i \Box G_j$.
So, arcs of type $i$ and $j$ correspond to arcs of $G_i$ and $G_j$, respectively, and have the associated labels.
For $k\ge 3$, the Cartesian product $G_1\Box G_2\Box \cdots\Box G_k$ is defined recursively as $((G_1\Box G_2)\Box \cdots )\Box G_k$. This Cartesian product is commutative and associative, as can be verified easily and is a well-known fact for the undirected analogue. 

Since we are particularly interested in synchronising arcs, we modify the Cartesian product $G_i\Box G_j$ according to the existence of synchronising arcs, i.e., pairs of arcs with the same label pair, with one arc in $G_i$ and one arc in $G_j$. 

The first step in this modification consists of ignoring (in fact deleting) the synchronising arcs while forming arcs in the product, but additionally combining pairs of synchronising arcs of $G_i$ and $G_j$ into one arc, yielding the {\em intermediate product\/} which we denote by $G_i \boxtimes G_j$. 
An example of the intermediate product is given in Figure~\ref{BiPartiteExampleDecomposition2}.

To be more precise, $G_i \boxtimes G_j$ is obtained from $G_i\Box G_j$ by first ignoring all except for the so-called {\em asynchronous\/} arcs, i.e., by only maintaining all arcs $a\in A_{i,j}$ for which $\mu(a)=((v_i,v_j),(w_i,w_j))$, whenever $v_j=w_j$ and $\lambda(a) \notin L_j$, as well as all arcs $a\in A_{i,j}$ for which $\mu(a)=((v_i,v_j),(w_i,w_j))$, whenever $v_i=w_i$ and $\lambda(a)\notin L_i$. 
Additionally, we add arcs that replace synchronising pairs $a_i\in A_i$ and $a_j\in A_j$ with $\lambda(a_i)=\lambda(a_j)$. If $\mu(a_i)=(v_i,w_i)$ and $\mu(a_j)=(v_j,w_j)$, such a pair is replaced by an arc $a_{i,j}$ with  $\mu(a_{i,j})=((v_i,v_j),(w_i,w_j))$ and $\lambda(a_{i,j})=\lambda(a_i)$. 
We call such arcs  of $G_i \boxtimes G_j$ {\em synchronous\/} arcs. 
The second step in this modification consists of removing (from $G_i \boxtimes G_j$) 
the 
vertices $(v_i,v_j)\in V_{i,j}$ together with the arcs $a$ with $tail(a)=(v_i,v_j)$ and the arcs $b$ with $head(b)=(v_i,v_j)$
for which $(v_i,v_j)$ has $in-degree > 0$ in $G_i\Box G_j$ but $in-degree = 0$ in $G_i\boxtimes G_j$. 
The removal of these 
vertices is then repeated in the newly obtained graph, and so on, until there are no more vertices with $in-degree = 0$ in the current graph with $in-degree > 0$ in $G_i\Box G_j$ and there are no more vertices with $out-degree = 0$ in the current graph with $out-degree > 0$ in $G_i\Box G_j$ . This finds its motivation in the fact that in our applications, the states that are represented by such vertices can never be reached, so are irrelevant. 

The resulting graph is called the {\em vertex-removing synchronised product\/} (VRSP for short) of $G_i$ and $G_j$, and denoted as $G_i \boxbackslash G_j$. 
For $k\ge 3$, the {VRSP} $G_1 \boxbackslash G_2 \boxbackslash \cdots \boxbackslash G_k$ is defined recursively as $((G_1 \boxbackslash G_2) \boxbackslash \cdots ) \boxbackslash G_k$. The VRSP is commutative, but not associative in general, in contrast to the Cartesian product. These properties are not relevant for the decomposition results that follow. 
However, for these results it is relevant to introduce counterparts of graph isomorphism and graph contraction that apply to our types of graphs. We define these counterparts in the next section.

\section{Graph isomorphism and graph contraction}\label{Terminology_morphisms}
The isomorphism we introduce in this section is an analogue of a known concept for unlabelled graphs, but involves statements on the labels.

We assume that two different arcs with the same tail and head have different labels; otherwise, we replace such multiple arcs by one arc with that label, because these arcs represent exactly the same action at the same stage of a  process.

Formally, an {\em isomorphism\/} from a graph $G$ to a graph $H$ consists of two bijections $\phi : V(G)\rightarrow V(H)$ and $\rho : A(G)\rightarrow A(H)$ such that for all $a \in A(G)$, one has $\mu(a) = (u, v)$ if and only if $\mu(\rho(a))=(\phi(u),\phi(v))$ and $\lambda(a)=\lambda(\rho(a))$. 
Since we assume that two different arcs with the same tail and head have different labels, however, the bijection $\rho$ is superfluous. The reason is that, if $(\phi,\rho)$ is an isomorphism, then $\rho$ is completely determined by $\phi$ and the labels. 
In fact, if $(\phi,\rho)$ is an isomorphism and $\mu(a)=(u,v)$ for an arc $a\in A(G)$, then $\rho(a)$ is the unique arc $b\in A(H)$ with $\mu(b)=(\phi(u),\phi(v))$ and  label $\lambda(b)=\lambda(a)$. 
Thus, we may define an isomorphism from $G$ to $H$ as a bijection $\phi : V(G)\rightarrow V(H)$ such that there exists an arc $a \in A(G)$ with $\mu(a)=(u,v)$ if and only if there exists an arc $b\in A(H)$ with $\mu(b)=(\phi(u),\phi(v))$ and $\lambda(b)=\lambda(a)$.
An isomorphism from $G$ to $H$ is denoted as $G\cong H$.

Next, we define what we mean by contraction.  
Let $X$ be a nonempty proper subset of $V(G)$, and let $Y=V(G)\setminus X$. 
By {\em contracting $X$\/} we mean replacing $X$ by a new vertex $\tilde{x}$, deleting all arcs with both ends in $X$, replacing each arc $a\in A(G)$ with $\mu(a)=(u,v)$ for $u\in X$ and $v\in Y$ by an arc $c$ with $\mu(c)=(\tilde{x},v)$ and $\lambda (c)=\lambda(a)$, and replacing each arc $b\in A(G)$ with $\mu(b)=(u,v)$ for $u\in Y$ and $v\in X$ by an arc $d$ with $\mu(d)=(u,\tilde{x})$ and $\lambda (d)=\lambda(b)$. We denote the resulting graph as $G/X$, and say that $G/X$ is the {\em contraction of $G$ with respect to $X$\/}. 
If we have a series of contractions of $G$ with respect to $X_1,\ldots,X_n$, $G/X_1/\ldots/X_n$, we denote the resulting graph as $G/_{i=1}^nX_i$.
When $X_i \cap X_j\neq \emptyset, i<j,$ then due to the contraction with respect to $X_i$ the vertices of $X_i$ are replaced by $\tilde{x}_i$ and therefore the vertices $X_i \cap X_j$ of $X_j$ are also replaced by $\tilde{x}_i$.
Hence, $X_{j}$ is a subset of the vertex set of the graph constructed by $G/X_1/\ldots/X_{j-1}$.



Finally, we recall the two decomposition theorems that were introduced in~\cite{dam} and modified in~\cite{ejgta1} (Note that if we would allow $X_2$ to be empty then in the case that $X_2$ is empty Theorem~\ref{theorem_2} is identical to Theorem~\ref{theorem_1}.).
\begin{theorem}[\cite{ejgta1}]\label{theorem_1}
Let $G$ be a graph, let $X$ be a nonempty proper subset of $V(G)$, and let $Y=V(G)\setminus X$. Suppose that each largest subset of arcs with the same label of $[X,Y]$ arc-induces a complete bipartite subgraph of $G$ and that the arcs of $G/X$ and $G/Y$ corresponding to the arcs of $[X,Y]$ are the only synchronising arcs of $G/X$ and $G/Y$. 
If $S'(G)\subseteq X$ and $[X,Y]$ has no backward arcs, then $G\cong G/Y\boxbackslash G/X$.
\end{theorem}
\begin{theorem}[\cite{ejgta1}]\label{theorem_2}
Let $G$ be a graph, and let $X_1$, $X_2$ and $Y=V(G)\setminus (X_1\cup X_2)$ be three disjoint nonempty subsets of $V(G)$. Suppose that each largest subset of arcs with the same label of $[X_1,Y]$ arc-induces a complete bipartite subgraph of $G$, each largest subset of arcs with the same label of $[Y,X_2]$ arc-induces a complete bipartite subgraph of $G$, the arcs of $[X_1,X_2]$ have no labels in common with any arc in $[X_1,Y]\cup[Y,X_2]$, and the arcs of $G/X_1/X_2$ and $G/Y$ corresponding to the arcs of $[X_1,Y]\cup [Y,X_2]\cup [X_1,X_2]$ are the only synchronising arcs of $G/X_1/X_2$ and $G/Y$. 
If $S'(G)\subseteq X_1$, and $[X_1,Y]$, $[Y,X_2]$ and $[X_1,X_2]$ have no backward arcs, then $G\cong G/Y\boxbackslash G/X_1/X_2$.
\end{theorem}

\section{The third, fourth and fifth graph-decomposition theorem.}\label{sec:decomp}
We assume that the graphs we want to decompose are connected; if not, we can apply our decomposition results to the components separately. 
We continue with presenting and proving our third decomposition theorem, given in Theorem~\ref{theorem_5}, of which an illustrative example is given in Figure~\ref{BiPartiteExampleDecomposition1}. 
In the third decomposition theorem we are going to decompose a graph $G$ that contains semicomplete bipartite subgraphs.
We continue with the decomposition of a graph $G$ where each subgraph of $G$ arc-induced by a set of all arcs with the same label in $G$ is a semicomplete bipartite subgraph $B(X_{i},Y_{j})$ of $G$.
The decomposition of $G$ consists of decomposing each semicomplete bipartite subgraph $B(X_i,Y_i)$ of $G=\overundercup{i=1}{n}B(X_i,Y_i)$ in such a manner that each $B(X_i,Y_i)$ is decomposed into two semicomplete bipartite graphs. 
We give a simple example of this decomposition in Figure~\ref{BiPartiteExampleDecomposition}, where we have nine semicomplete bipartite subgraph $B(X_i,Y_i)$ of which eight subgraphs are trivial bipartite subgraphs.
Because with respect to the VRSP a trivial bipartite subgraph $B(X,Y)$ is idempotent, $B(X,Y)\cong B(X,Y)\boxbackslash B(X,Y)$, we do not contract these subgraphs in the example depicted in Figure~\ref{BiPartiteExampleDecomposition}.
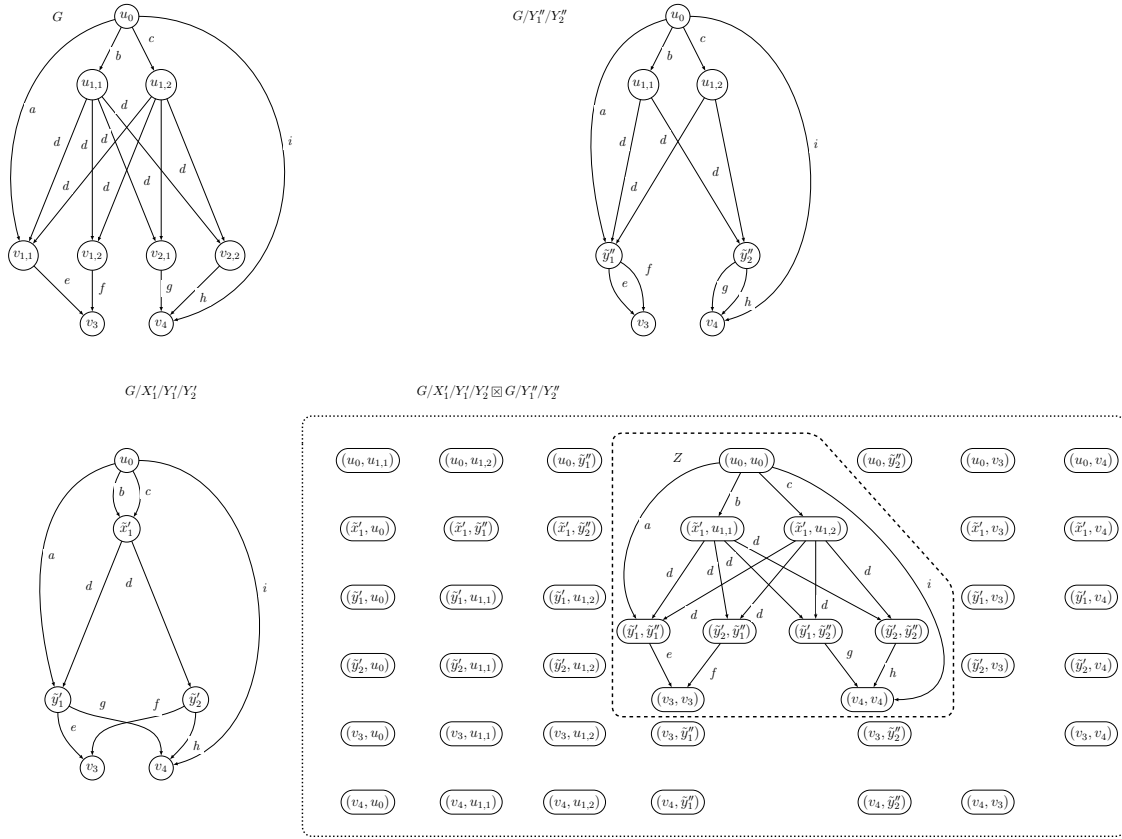
\begin{figure}[H]
\begin{center}
\resizebox{1.0\textwidth}{!}{
\begin{tikzpicture}[->,>=latex,shorten >=0pt,auto,node distance=2.5cm,
  main node/.style={circle,fill=blue!10,draw, font=\sffamily\Large\bfseries}]
  \tikzset{VertexStyle/.append style={
  font=\itshape\large, shape = circle,inner sep = 2pt, outer sep = 0pt,minimum size = 20 pt,draw}}
  \tikzset{EdgeStyle/.append style={thin}}
  \tikzset{LabelStyle/.append style={font = \itshape}}
  \SetVertexMath
  \def\x{0.0}
  \def\y{2.0}
\node at (\x-10,\y-0) {$G$};
\node at (\x+4.0,\y-0) {$G/Y''_1/Y''_2$};
\node at (\x-7.0,\y-11) {$G/X'_1/Y'_1/Y'_2$};
\node at (\x+2.5,\y-11) {$G/X'_1/Y'_1/Y'_2\boxtimes G/Y''_1/Y''_2$};
\node at (\x+8,\y-12.9) {$Z$};
  \def\x{-10.0}
  \def\y{-2.0}
  \Vertex[x=\x+2, y=\y+4.0,L={u_{0}}]{u_1}
  \Vertex[x=\x+1, y=\y+2.0,L={u_{1,1}}]{u_2}
  \Vertex[x=\x+3, y=\y+2.0,L={u_{1,2}}]{u_3}
  \Vertex[x=\x-1, y=\y-3,L={v_{1,1}}]{v_1}
  \Vertex[x=\x+1, y=\y-3,L={v_{1,2}}]{v_2}
  \Vertex[x=\x+3, y=\y-3,L={v_{2,1}}]{v_3}
  \Vertex[x=\x+5, y=\y-3,L={v_{2,2}}]{v_4}
  \Vertex[x=\x+1, y=\y-5,L={v_{3}}]{v_5}
  \Vertex[x=\x+3, y=\y-5,L={v_{4}}]{v_6}

  \Edge[label = a, labelstyle={xshift=0pt, yshift=2pt}, style={in = 105, out = 195,min distance=2cm}](u_1)(v_1) 
  \Edge[label = i, labelstyle={xshift=0pt, yshift=2pt}, style={in = 15, out = 0,min distance=5cm}](u_1)(v_6) 
  \Edge[label = b, labelstyle={xshift=0pt, yshift=2pt}](u_1)(u_2) 
  \Edge[label = c, labelstyle={xshift=0pt, yshift=2pt}](u_1)(u_3) 
  \Edge[label = d, labelstyle={xshift=-8pt, yshift=32pt}](u_2)(v_1) 
  \Edge[label = d, labelstyle={xshift=-14pt, yshift=22pt}](u_2)(v_2) 
  \Edge[label = d, labelstyle={xshift=-26pt, yshift=19pt}](u_2)(v_3) 
  \Edge[label = d, labelstyle={xshift=-38pt, yshift=47pt}](u_2)(v_4) 
 
  \Edge[label = d, labelstyle={xshift=-29pt, yshift=-5pt}](u_3)(v_1) 
  \Edge[label = d, labelstyle={xshift=-24pt, yshift=-6pt}](u_3)(v_2) 
  \Edge[label = d, labelstyle={xshift=-19pt, yshift=-8pt}](u_3)(v_3) 
  \Edge[label = d, labelstyle={xshift=-18pt, yshift=-6pt}](u_3)(v_4) 

  \Edge[label = e, labelstyle={xshift=0pt, yshift=2pt}](v_1)(v_5)
  \Edge[label = f, labelstyle={xshift=2pt, yshift=2pt}](v_2)(v_5) 
  \Edge[label = g, labelstyle={xshift=0pt, yshift=2pt}](v_3)(v_6) 
  \Edge[label = h, labelstyle={xshift=0pt, yshift=2pt}](v_4)(v_6)

   \def\x{6.0}
  \def\y{-2.0}
  \Vertex[x=\x+2, y=\y+4.0,L={u_{0}}]{u_1}
  \Vertex[x=\x+1, y=\y+2.0,L={u_{1,1}}]{u_2}
  \Vertex[x=\x+3, y=\y+2.0,L={u_{1,2}}]{u_3}
  \Vertex[x=\x-0, y=\y-3,L={\tilde{y}''_{1}}]{v_1}
  \Vertex[x=\x+4, y=\y-3,L={\tilde{y}''_{2}}]{v_3}
  \Vertex[x=\x+1, y=\y-5,L={v_{3}}]{v_5}
  \Vertex[x=\x+3, y=\y-5,L={v_{4}}]{v_6}

  \Edge[label = a, labelstyle={xshift=0pt, yshift=2pt}, style={in = 105, out = 195,min distance=2cm}](u_1)(v_1) 
  \Edge[label = i, labelstyle={xshift=0pt, yshift=2pt}, style={in = 15, out = 0,min distance=4cm}](u_1)(v_6) 
  \Edge[label = b, labelstyle={xshift=0pt, yshift=2pt}](u_1)(u_2) 
  \Edge[label = c, labelstyle={xshift=0pt, yshift=2pt}](u_1)(u_3) 
  \Edge[label = d, labelstyle={xshift=-10pt, yshift=32pt}](u_2)(v_1) 
  \Edge[label = d, labelstyle={xshift=-34pt, yshift=19pt}](u_2)(v_3) 
 
  \Edge[label = d, labelstyle={xshift=-29pt, yshift=-5pt}](u_3)(v_1) 
  \Edge[label = d, labelstyle={xshift=-19pt, yshift=-8pt}](u_3)(v_3) 

  \Edge[label = e, labelstyle={xshift=0pt, yshift=2pt}, style={in = 135, out = -90,min distance=0cm}](v_1)(v_5)
  \Edge[label = f, labelstyle={xshift=2pt, yshift=2pt}, style={in = 90, out = -30,min distance=0cm}](v_1)(v_5) 
  \Edge[label = g, labelstyle={xshift=0pt, yshift=2pt}, style={in = 90, out = 210,min distance=0cm}](v_3)(v_6) 
  \Edge[label = h, labelstyle={xshift=0pt, yshift=2pt}, style={in = 45, out = -90,min distance=0cm}](v_3)(v_6)

    \def\x{-10.0}
  \def\y{-15.0}
  \Vertex[x=\x+2, y=\y+4.0,L={u_{0}}]{u_1}
  \Vertex[x=\x+2, y=\y+2.0,L={\tilde{x}'_{1}}]{u_2}
  \Vertex[x=\x-0, y=\y-3,L={\tilde{y}'_{1}}]{v_1}
  \Vertex[x=\x+4, y=\y-3,L={\tilde{y}'_{2}}]{v_3}
  \Vertex[x=\x+1, y=\y-5,L={v_{3}}]{v_5}
  \Vertex[x=\x+3, y=\y-5,L={v_{4}}]{v_6}

  \Edge[label = a, labelstyle={xshift=0pt, yshift=2pt}, style={in = 105, out = 195,min distance=2cm}](u_1)(v_1) 
  \Edge[label = i, labelstyle={xshift=0pt, yshift=2pt}, style={in = 15, out = 0,min distance=4cm}](u_1)(v_6) 
  \Edge[label = b, labelstyle={xshift=0pt, yshift=2pt}, style={in = 120, out = -120,min distance=0cm}](u_1)(u_2) 
  \Edge[label = c, labelstyle={xshift=0pt, yshift=2pt}, style={in = 60, out = -60,min distance=0cm}](u_1)(u_2) 
  \Edge[label = d, labelstyle={xshift=-10pt, yshift=32pt}](u_2)(v_1) 
  \Edge[label = d, labelstyle={xshift=-34pt, yshift=19pt}](u_2)(v_3) 

  \Edge[label = e, labelstyle={xshift=0pt, yshift=2pt}, style={in = 135, out = -90,min distance=0cm}](v_1)(v_5)
  \Edge[label = g, labelstyle={xshift=-28pt, yshift=8pt}, style={in = 90, out = -30,min distance=0cm}](v_1)(v_6) 
  \Edge[label = f, labelstyle={xshift=20pt, yshift=24pt}, style={in = 90, out = 210,min distance=0cm}](v_3)(v_5) 
  \Edge[label = h, labelstyle={xshift=0pt, yshift=2pt}, style={in = 45, out = -90,min distance=0cm}](v_3)(v_6)

  \tikzset{VertexStyle/.append style={
  font=\itshape\large, shape = rounded rectangle, inner sep = 2pt, outer sep = 0pt,minimum size = 20 pt,draw}}

  \def\x{10.0}
  \def\y{-13.0}
  \Vertex[x=\x+0, y=\y+2.0,L={(u_0,u_0)}]{u0u0}
  \Vertex[x=\x+3-14, y=\y+2.0,L={(u_0,u_{1,1})}]{u0u11}
  \Vertex[x=\x+6-14, y=\y+2.0,L={(u_0,u_{1,2})}]{u0u12}
  \Vertex[x=\x+9-14, y=\y+2.0,L={(u_0,\tilde{y}''_{1})}]{u0y1}
  \Vertex[x=\x+12-8, y=\y+2.0,L={(u_0,\tilde{y}''_{2})}]{u0y2}
  \Vertex[x=\x+15-8, y=\y+2.0,L={(u_0,v_{3})}]{u0v3}
  \Vertex[x=\x+18-8, y=\y+2.0,L={(u_0,v_{4})}]{u0v4}
  \def\x{6.0}
  \def\y{-15.0}
  \Vertex[x=\x+0-7, y=\y+2.0,L={(\tilde{x}'_{1},u_0)}]{x1u0}
  \Vertex[x=\x+3, y=\y+2.0,L={(\tilde{x}'_{1},u_{1,1})}]{x1u11}
  \Vertex[x=\x+6, y=\y+2.0,L={(\tilde{x}'_{1},u_{1,2})}]{x1u12}
  \Vertex[x=\x+9-13, y=\y+2.0,L={(\tilde{x}'_{1},\tilde{y}''_{1})}]{x1y1}
  \Vertex[x=\x+12-13, y=\y+2.0,L={(\tilde{x}'_{1},\tilde{y}''_{2})}]{x1y2}
  \Vertex[x=\x+15-4, y=\y+2.0,L={(\tilde{x}'_{1},v_{3})}]{x1v3}
  \Vertex[x=\x+18-4, y=\y+2.0,L={(\tilde{x}'_{1},v_{4})}]{x1v4}
  \def\x{-1.0}
  \def\y{-17.0}
  \Vertex[x=\x+0, y=\y+2.0,L={(\tilde{y}'_{1},u_0)}]{y1u0}
  \Vertex[x=\x+3, y=\y+2.0,L={(\tilde{y}'_{1},u_{1,1})}]{y1u11}
  \Vertex[x=\x+6, y=\y+2.0,L={(\tilde{y}'_{1},u_{1,2})}]{y1u12}
  \Vertex[x=\x+8.0, y=\y+1.0,L={(\tilde{y}'_{1},\tilde{y}''_{1})}]{y1y1}
  \Vertex[x=\x+13.0, y=\y+1.0,L={(\tilde{y}'_{1},\tilde{y}''_{2})}]{y1y2}
  \Vertex[x=\x+18+0, y=\y+2.0,L={(\tilde{y}'_{1},v_{3})}]{y1v3}
  \Vertex[x=\x+21+0, y=\y+2.0,L={(\tilde{y}'_{1},v_{4})}]{y1v4}
  \def\x{0.5}
  \def\y{-19.0}
  \Vertex[x=\x+-1.5, y=\y+2.0,L={(\tilde{y}'_{2},u_0)}]{y2u0}
  \Vertex[x=\x+1.5, y=\y+2.0,L={(\tilde{y}'_{2},u_{1,1})}]{y2u11}
  \Vertex[x=\x+4.5, y=\y+2.0,L={(\tilde{y}'_{2},u_{1,2})}]{y2u12}
  \Vertex[x=\x+9, y=\y+3.0,L={(\tilde{y}'_{2},\tilde{y}''_{1})}]{y2y1}
  \Vertex[x=\x+14, y=\y+3.0,L={(\tilde{y}'_{2},\tilde{y}''_{2})}]{y2y2}
  \Vertex[x=\x+16.5+0, y=\y+2.0,L={(\tilde{y}'_{2},v_{3})}]{y2v3}
  \Vertex[x=\x+19.5+0, y=\y+2.0,L={(\tilde{y}'_{2},v_{4})}]{y2v4}
  \def\x{-1.0}
  \def\y{-21.0}
  \Vertex[x=\x+0, y=\y+2.0,L={(v_3,u_0)}]{v3u0}
  \Vertex[x=\x+3, y=\y+2.0,L={(v_3,u_{1,1})}]{v3u11}
  \Vertex[x=\x+6, y=\y+2.0,L={(v_3,u_{1,2})}]{v3u12}
  \Vertex[x=\x+9, y=\y+2.0,L={(v_3,\tilde{y}''_{1})}]{v3y1}
  \Vertex[x=\x+15, y=\y+2.0,L={(v_3,\tilde{y}''_{2})}]{v3y2}
  \Vertex[x=\x+9, y=\y+3,L={(v_3,v_{3})}]{v3v3}
  \Vertex[x=\x+21, y=\y+2,L={(v_3,v_{4})}]{v3v4}
  \def\x{-1.0}
  \def\y{-23.0}
  \Vertex[x=\x+0, y=\y+2.0,L={(v_4,u_0)}]{v4u0}
  \Vertex[x=\x+3, y=\y+2.0,L={(v_4,u_{1,1})}]{v4u11}
  \Vertex[x=\x+6, y=\y+2.0,L={(v_4,u_{1,2})}]{v4u12}
  \Vertex[x=\x+9, y=\y+2.0,L={(v_4,\tilde{y}''_{1})}]{v4y1}
  \Vertex[x=\x+15, y=\y+2.0,L={(v_4,\tilde{y}''_{2})}]{v4y2}
  \Vertex[x=\x+18, y=\y+2.0,L={(v_4,v_{3})}]{v4v3}
  \Vertex[x=\x+14.5, y=\y+5.0,L={(v_4,v_{4})}]{v4v4}

  \Edge[label = a, labelstyle={xshift=0pt, yshift=2pt}, style={in = 120, out = 185,min distance=2cm}](u0u0)(y1y1) 
  \Edge[label = b, labelstyle={xshift=0pt, yshift=2pt}](u0u0)(x1u11) 
  \Edge[label = i, labelstyle={xshift=0pt, yshift=2pt}, style={in = 0, out = -15,min distance=3.5cm}](u0u0)(v4v4) 
  \Edge[label = c, labelstyle={xshift=0pt, yshift=2pt}](u0u0)(x1u12) 
  \Edge[label = d, labelstyle={xshift=-14pt, yshift=14pt}](x1u11)(y1y1) 
  \Edge[label = d, labelstyle={xshift=-16pt, yshift=0pt}](x1u11)(y2y1) 
  \Edge[label = d, labelstyle={xshift=-36pt, yshift=-21pt}](x1u12)(y1y1) 
  \Edge[label = d, labelstyle={xshift=-18pt, yshift=-20pt}](x1u12)(y2y1) 
  \Edge[label = d, labelstyle={xshift=-50pt, yshift=26pt}](x1u11)(y2y2) 
  \Edge[label = d, labelstyle={xshift=-36pt, yshift=8pt}](x1u11)(y1y2) 
  \Edge[label = d, labelstyle={xshift=0pt, yshift=0pt}](x1u12)(y2y2) 
  \Edge[label = d, labelstyle={xshift=0pt, yshift=-20pt}](x1u12)(y1y2) 
  \Edge[label = e, labelstyle={xshift=0pt, yshift=2pt}](y1y1)(v3v3)
  \Edge[label = f, labelstyle={xshift=2pt, yshift=2pt}](y2y1)(v3v3) 
  \Edge[label = g, labelstyle={xshift=0pt, yshift=2pt}](y1y2)(v4v4) 
  \Edge[label = h, labelstyle={xshift=0pt, yshift=2pt}](y2y2)(v4v4)

  \def\x{6}
  \def\y{-12.0}
\draw[circle, -,dashed, very thick,rounded corners=8pt] (\x+0.1,\y+1.0)--(\x+0.1,\y+1.8)-- (\x+6.0,\y+1.8)-- (\x+10.0,\y-2.7)--(\x+10.0,\y-6.5)-- (\x+1.5,\y-6.5)-- (\x+0.1,\y-6.5) -- (\x+0.1,\y-0.5)--(\x+0.1,\y+1.0);

  \def\x{-3}
  \def\y{-11.5}
  \draw[circle, -,dotted, very thick,rounded corners=8pt] (\x+0.1,\y+1.0)--(\x+0.1,\y+1.8)-- (\x+24.0,\y+1.8)-- (\x+24.0,\y-10.5)-- (\x+1.5,\y-10.5)-- (\x+0.1,\y-10.5) -- (\x+0.1,\y-0.5)--(\x+0.1,\y+1.0);
%
%
\end{tikzpicture}
}
\end{center}
\caption{Decomposition of $G\cong G/X'_1/Y'_1/Y'_2\boxbackslash G/Y''_1/Y''_2$, $X'_1=\{u_{1,1},u_{1,2}\},Y'_1=\{v_{1,1},v_{2,1}\},Y'_2=\{v_{1,2},v_{2,2}\},Y''_1=\{v_{1,1},v_{1,2}\},Y''_2=\{v_{2,1},v_{2,2}\}$. The set $Z$ from the proof of Theorem~\ref{theorem_5} and the graph isomorphic to $G$ induced by $Z$ in $G/X'_1/Y'_1/Y'_2\boxtimes G/Y''_1/Y''_2$ is indicated within the dotted region.}
  \label{BiPartiteExampleDecomposition}
\end{figure}

To decompose a graph with respect to the decomposition of a non-trivial semicomplete bipartite subgraph of $G$ where each arc has the same label we have to decompose each of these non-trivial bipartite semicomplete subgraphs of $G$.
This is obvious, because if one of these subgraphs is not decomposed, say $B(X_1,X_2)$, the VRSP of the two decompositions $H$ and $I$ of $G$ will contain $B(X_1,X_2)\boxbackslash B(X_1,X_2)$. 
This subgraph has $|X_1|^2+|X_2|^2$ vertices in $H \boxbackslash I$ and therefore, $G\ncong H \boxbackslash  I$ for $|X_1|>1$ or $|X_2|>1$.
As mentioned before, if a subgraph induced by the set of all arcs with the same label in $G$ is a trivial bipartite subgraph $B(X_1,X_2)$, then this subgraph does not have to be decomposed because $B(X_1,X_2)\boxbackslash B(X_1,X_2)\cong B(X_1,X_2)$.
But in the proof of Theorem~\ref{theorem_5}, we decompose all semicomplete bipartite subgraphs of $G$.

For reasons we will clarify in Theorems~\ref{theorem_5} and~\ref{theorem_6}, we introduce the \emph{matrix graph}, the \emph{bipartite matrix graph} and the \emph{Cartesian matrix graph}.


We define $M$ as a two-dimensional index set with pairs of indices that are numbered in the following manner: $M=\{(i,j)\mid i\in I= \{1,\ldots,m\}, j\in J= \{1,\ldots,n\}\}$.
A graph $G$ of which the vertices are numbered according to the index set $M$ has sets of rows $R_i=\{v_{(i,j)}\mid j\in J\},i\in I$, and sets of columns $C_j=\{v_{(i,j)}\mid i\in I\},j\in J$.
For brevity, in the sequel we denote the vertices $v_{(i,j)}$ as $v_{i,j}$.

For a subgraph $G[X]$ of a graph $G$, we call $X$ a grid of vertices when the vertices of $X$ are numbered in the following manner. 
The vertices $v_{i,j}\in X$ are numbered such that $i\in I_X\subseteq I$ and $j\in J_X\subseteq J$, $|X|=|I_X|\cdot|J_X|=m_1\cdot n_1$, $1\leq m_1\leq m,1\leq n_1\leq n$.
Hence, $X=\{v_{i,j}\mid i\in I_{X}\subseteq I,j\in J_{X}\subseteq J\}$, $|I_{X}|=m_1,|J_{X}|=n_1$, with rows $X'_{i}\subseteq R_i,X'_{i}=\{v_{i,j}|j\in J_{X}\},i\in I_{X},$ and with columns $X''_{j}\subseteq C_j$, $X''_{j}=\{v_{i,j}|i\in I_{X}\},j\in J_{X}$.
In the example given in Figure~\ref{BiPartiteExampleDecomposition1}, each of the sets $X_1,\ldots,X_4$ is a grid.


A \emph{matrix graph} $G$ is a graph $G$ for which the vertices are numbered according to a subset $M'$ of the index set $M$.

A \emph{bipartite matrix graph} $G$ is a matrix graph $G$ consisting solely of $x$ bipartite subgraphs where each bipartite subgraph has arcs with identical labels and each pair of such bipartite subgraphs do not share a label.
Therefore, we require, firstly, that the bipartite matrix graph $G$ is a matrix graph consisting of $x$ bipartite subgraphs $B(X_i,X_j)$ and $z$ not necessarily disjunct sets $X_k$ where $X_k=X_i$ or $X_k=X_j$ and $z\leq 2x$.
Secondly, all subgraphs of $G$ arc-induced by a set of all arcs with identical labels are semicomplete bipartite subgraphs $B(X_i,X_j)$ of $G$, all $X_i$ and $X_j$ are grids of vertices, $i,j\in\chi=\{1,\ldots,z\},i\neq j$, and $[X_i,X_j]$ contains only forward arcs or $[X_i,X_j]$ contains only backward arcs.
Thirdly, we require that whenever a row $X'_{k,x}$ of the set $X_k$ and a row $X'_{l,y}$ of the set $X_l$ share a vertex $v_{i,j}$ then $X'_{k,x}\subseteq R_i$ and $X'_{l,y}\subseteq R_i$, $k,l\in\chi$.
Fourthly, let $R'_i\subseteq V(G)\subseteq R_i$. Then for any division of $R'_i$ into the sets $R'_{i_1}$ and $R'_{i_2}$, $R'_i=R'_{i_1}\cup R'_{i_2}$, there is always a row $X'_{k,x}\subseteq R_{i_1}$ and a row $X'_{l,y}\subseteq R'_{i_2}$ with $X'_{k,x}\cap X'_{l,y}\neq \emptyset$.
Fifthly, we require that whenever a column $X''_{k,x}$ of the set $X_k$ and a column $X''_{l,y}$ of the set $X_l$ share a vertex $v_{i,j}$ then $X''_{k,x}\subseteq C_j$ and $X''_{l,y}\subseteq C_j$, $k,l\in\chi$.
Sixthly, let $C'_j\subseteq V(G)\subseteq C_j$. Then for any division of $C'_j$ into the sets $C'_{j_1}$ and $C'_{j_2}$, $C'_j=C'_{j_1}\cup C'_{j_2}$, there is always a column $X''_{k,x}\subseteq C'_{j_1}$ and a column $X''_{l,y}\subseteq C'_{j_2}$ with $X''_{k,x}\cap X''_{l,y}\neq \emptyset$.
We call a graph $G$ that fulfils these six requirements a \emph{bipartite matrix graph}.

The purpose of the bipartite matrix graph is that after the decomposition of any subgraph $B(X_i,X_j)$ of the bipartite matrix graph $G$, into graphs $B(X'_i,X'_j)$ and $B(X''_i,X''_j)$ with $B(X_i,X_j)\cong B(X'_i,X'_j)\boxbackslash B(X''_i,X''_j)$ by Theorem~\ref{theorem_5}, we have that all vertices $v_{i,x}\in V(B(X_i,X_j))$ are replaced by the vertex $\tilde{x}_i\in V(B(X'_i,X'_j))$ and all vertices $v_{x,j}\in V(B(X_i,X_j))$ are replaced by the vertex $\tilde{x}_j\in V(B(X''_i,X''_j))$.
With the third and fourth requirement, we assure that all vertices in the rows of $R_i$ must have the same first index and vertices not in the rows of $R_i$ have a different first index.
With the fifth and sixth requirement, we assure that all vertices in the columns of $C_j$ must have the same second index and vertices not in the columns of $C_j$ have a different second index.

A \emph{Cartesian matrix graph} $G$ is a matrix graph with rows $R_i, i\in I_i\subseteq I$ and columns $C_j, j\in J_j\subseteq J$, for which $G[R_x]\cong G[R_y],x,y\in I_i$, $G[C_{x}]\cong  G[C_{y}],x,y\in J_j$, and the arcs of $G[R_i]$ and the arcs of $G[C_j]$ have no labels in common, if $a$ is an arc of $A(G)$ with $\mu(a)=uv$ then $u,v\in R_i$ or $u,v\in C_j$.

In Figure~\ref{BiPartiteExampleDecomposition1}, we have depicted the vertex sets $X_i$ of the bipartite matrix graph $G$ comprising the bipartite semicomplete subgraphs $B(X_i,X_4)$ for $i=1,\ldots,3,$ where the labels of the arcs of $B(X_i,X_4)$ are the same and the labels of the arcs of $B(X_i,X_4)$ and $B(X_j,X_4), i\neq j$, are different.
All vertex sets $X_i$ are grids.
The arcs connected to the dotted box, dashdotted box and straight boxes are connected to the vertices these boxes contain. 
For example, the straight arcs with label $d$ connected to the boxes of vertex set $X_1$ represent the arc set $\{u_{2,2}u_{7,7},u_{2,4}u_{7,7},u_{2,5}u_{7,7},u_{5,2}u_{7,7},u_{5,4}u_{7,7},u_{5,5}u_{7,7},u_{6,2}u_{7,7},u_{6,4}u_{7,7}$, $u_{6,5}u_{7,7}\}$ of arcs with label $c$.  
Furthermore, due to the contraction of the second row of $X_2$ the vertices $u_{2,1},\ldots,u_{2,4}$ are replaced by $\tilde{x}_2$, which gives a new first row of $X_1$ consisting of the vertices $\tilde{x}_2$ and $u_{2,5}$. 
Later on, by contraction of the first row of $X_1$, the vertices $\tilde{x}_2$ and $u_{2,5}$ are replaced by $\tilde{x}_2$.
In Figure~\ref{BiPartiteExampleDecomposition2}, we have depicted the graph $G/_{i=1}^3X'_{1,i}/_{i=1}^4X'_{2,i}/_{i=1}^4X'_{3,i}$ $/X'_{4,1}\boxtimes G/_{i=1}^3X''_{1,i}/_{i=1}^5X''_{2,i}/_{i=1}^3X''_{3,i}/X''_{4,1}$ which is isomorphic to the graph $G$ of Figure~\ref{BiPartiteExampleDecomposition1} after deletion of the vertices with in-degree zero in $G/_{i=1}^3X'_{1,i}/_{i=1}^4X'_{2,i}/_{i=1}^4X'_{3,i}$ $/X'_{4,1}\boxtimes G/_{i=1}^3X''_{1,i}/_{i=1}^5X''_{2,i}/_{i=1}^3X''_{3,i}/X''_{4,1}$ and in-degree greater than zero in $G/_{i=1}^3X'_{1,i}/_{i=1}^4X'_{2,i}/_{i=1}^4X'_{3,i}$ $/X'_{4,1}\Box G/_{i=1}^3X''_{1,i}/_{i=1}^5X''_{2,i}/_{i=1}^3X''_{3,i}/X''_{4,1}$.
Furthermore, because the pairwise intersection of the grids $X_1,X_2$ and $X_3$ are grids, the graph $G$ is isomorphic to the graph $G/_{i=1}^3X'_{1,i}/_{i=1}^4X'_{2,i}$ $/_{i=1}^4X'_{3,i}/X'_{4,1}\boxbackslash G/_{i=1}^3X''_{1,i}/_{i=1}^5X''_{2,i}/_{i=1}^3X''_{3,i}/X''_{4,1}$, which we will prove in Theorem~\ref{theorem_5}.
Due to the numbering scheme of the vertices in $V(G)$ we have that $G/_{i=1}^3X'_{1,i}/_{i=1}^4X'_{2,i}/_{i=1}^4X'_{3,i}/X'_{4,1}\boxbackslash G/_{i=1}^3X''_{1,i}/_{i=1}^5X''_{2,i}/_{i=1}^3X''_{3,i}/X''_{4,1}\cong G/_{i=1}^7R_{i}\boxbackslash G/_{i=1}^7C_{i}\cong G$.
In Theorem~\ref{theorem_5}, we use the notation with the sets $X_i$ and in Theorem~\ref{theorem_6}, we use the notation with the rows $R_i$ and the columns $C_i$.

\begin{figure}[H]
\begin{center}
\resizebox{1\textwidth}{!}{
\begin{tikzpicture}[->,>=latex,shorten >=0pt,auto,node distance=2.5cm,
  main node/.style={circle,fill=blue!10,draw, font=\sffamily\Large\bfseries}]
  \tikzset{VertexStyle/.append style={
  font=\itshape\large, shape = circle,inner sep = 2pt, outer sep = 0pt,minimum size = 20 pt,draw}}
  \tikzset{EdgeStyle/.append style={thin}}
  \tikzset{LabelStyle/.append style={font = \itshape}}
  \SetVertexMath

  \clip (-2,2) rectangle (18, -19.0);

  \def\x{0.0}
  \def\y{2.0}
\node at (\x-1.0,\y-0.75) {$G$};
  \def\x{0.0}
  \def\y{0.0}
  \Vertex[x=\x+0, y=\y+0,L={u_{1,1}}]{u_00}
  \Vertex[x=\x+2, y=\y+0,L={u_{1,2}}]{u_01}
  \Vertex[x=\x+4, y=\y+0,L={u_{1,3}}]{u_02}
  \Vertex[x=\x+6, y=\y+0,L={u_{1,4}}]{u_03}
  \def\y{-2.0}
  \Vertex[x=\x+0, y=\y+0,L={u_{2,1}}]{u_10}
  \Vertex[x=\x+2, y=\y+0,L={u_{2,2}}]{u_11}
  \Vertex[x=\x+4, y=\y+0,L={u_{2,3}}]{u_12}
  \Vertex[x=\x+6, y=\y+0,L={u_{2,4}}]{u_13}
  \Vertex[x=\x+8, y=\y+0,L={u_{2,5}}]{u_14}
  \def\y{-4.0}
  \Vertex[x=\x+0, y=\y+0,L={u_{3,1}}]{u_20}
  \Vertex[x=\x+2, y=\y+0,L={u_{3,2}}]{u_21}
  \Vertex[x=\x+4, y=\y+0,L={u_{3,3}}]{u_22}
  \Vertex[x=\x+6, y=\y+0,L={u_{3,4}}]{u_23}
  \def\y{-6.0}
  \Vertex[x=\x+0, y=\y+0,L={u_{4,1}}]{u_30}
  \Vertex[x=\x+2, y=\y+0,L={u_{4,2}}]{u_31}
  \Vertex[x=\x+4, y=\y+0,L={u_{4,3}}]{u_32}
  \Vertex[x=\x+6, y=\y+0,L={u_{4,4}}]{u_33}
  \Vertex[x=\x+10, y=\y+0,L={u_{4,6}}]{u_35}
  \def\y{-8.0}
  \Vertex[x=\x+0, y=\y+0,L={u_{5,1}}]{u_40}
  \Vertex[x=\x+2, y=\y+0,L={u_{5,2}}]{u_41}
  \Vertex[x=\x+4, y=\y+0,L={u_{5,3}}]{u_42}
  \Vertex[x=\x+6, y=\y+0,L={u_{5,4}}]{u_43}
  \Vertex[x=\x+8, y=\y+0,L={u_{5,5}}]{u_44}
  \Vertex[x=\x+10, y=\y+0,L={u_{5,6}}]{u_45}
  \def\y{-10.0}
  \Vertex[x=\x+2, y=\y+0,L={u_{6,2}}]{u_51}
  \Vertex[x=\x+4, y=\y+0,L={u_{6,3}}]{u_52}
  \Vertex[x=\x+6, y=\y+0,L={u_{6,4}}]{u_53}
  \Vertex[x=\x+8, y=\y+0,L={u_{6,5}}]{u_54}
  \Vertex[x=\x+10, y=\y+0,L={u_{6,6}}]{u_55}

  \def\y{-13.0-2}
  \Vertex[x=\x+0, y=\y+0,L={\tilde{x}''_{1}}]{x_0}
  \Vertex[x=\x+2, y=\y+0,L={\tilde{x}''_{2}}]{x_1}
  \Vertex[x=\x+4, y=\y+0,L={\tilde{x}''_{3}}]{x_2}
  \Vertex[x=\x+6, y=\y+0,L={\tilde{x}''_{4}}]{x_3}
  \Vertex[x=\x+8, y=\y+0,L={\tilde{x}''_{5}}]{x_4}
  \Vertex[x=\x+10, y=\y+0,L={\tilde{x}''_{6}}]{x_5}

  \def\x{2}
  \def\y{-10.0}
  \Vertex[x=\x+12, y=\y+0,L={\tilde{x}'_{6}}]{y_5}
  \Vertex[x=\x+12, y=\y+2,L={\tilde{x}'_{5}}]{y_4}
  \Vertex[x=\x+12, y=\y+4,L={\tilde{x}'_{4}}]{y_3}
  \Vertex[x=\x+12, y=\y+6,L={\tilde{x}'_{3}}]{y_2}
  \Vertex[x=\x+12, y=\y+8,L={\tilde{x}'_{2}}]{y_1}
  \Vertex[x=\x+12, y=\y+10,L={\tilde{x}'_{1}}]{y_0}
  
\def\x{-4}
\def\y{-6.0}
\Vertex[x=\x+16, y=\y-6,L={u_{7,7}}]{v_0}
\def\x{16-2}
\def\y{-5.0-7}
\Vertex[x=\x, y=\y,L={\tilde{x}_{7}}]{y'_0}
\def\x{5+7}
\def\y{-15.0+2}
\Vertex[x=\x, y=\y-2,L={\tilde{x}_{7}}]{y'_1}

\def\x{-4}
\def\y{-6.0}	
\Edge[label = a, labelstyle={xshift=0pt, yshift=-2pt}, style={very thick, dashdotted, in = 60, out = 30,min distance=5cm}](\x+5.2+4.75,\y+1.0)(v_0)
\Edge[label = a, labelstyle={xshift=0pt, yshift=-2pt}, style={very thick, dashdotted, in = 90, out = 0,min distance=2cm}](\x+13.2+1.75,\y-0.5)(v_0)
\Edge[label = b, labelstyle={xshift=0pt, yshift=-0pt}, style={thick,dotted, in = 210, out = 270,min distance=5.5cm}](\x+4.4,\y-3.1)(v_0) 
\Edge[label = c, labelstyle={xshift=0+4pt, yshift=-4pt}, style={thick, in = 60, out = 20,min distance=8cm}](\x+5.5+1.1,\y+4.8)(v_0) 
\Edge[label = c, labelstyle={xshift=-20pt, yshift=12pt}, style={thick, in = 75, out = 0,min distance=1cm}](\x+9.5+3.4,\y+4.8-0.4)(v_0)
\Edge[label = c, labelstyle={xshift=8pt, yshift=+6pt}, style={thick, in = 195, out = 300,min distance=-2cm}](\x+5.5+0.5,\y+4.8-9.6-0.2)(v_0) 
\Edge[label = c, labelstyle={xshift=-28pt, yshift=10pt}, style={thick, in = 180, out = 300,min distance=-2cm}](\x+9.5+2,\y+4.8-9.6-0.2)(v_0)

\def\x{-4}
\def\y{-6.0}	
\Edge[label = \{b\}, labelstyle={xshift=-185pt, yshift=70pt}, style={in = 270, out = 270,min distance=2cm}](x_0)(y'_1) 
\Edge[label = \{a\textsf{,}b\textsf{,}c\}, labelstyle={xshift=-166pt, yshift=53pt}, style={in = 270, out = 270,min distance=2cm}](x_1)(y'_1)
\Edge[label = \{a\textsf{,}b\}, labelstyle={xshift=-130pt, yshift=37pt}, style={in = 270, out = 270}](x_2)(y'_1) 
\Edge[label = \{a\textsf{,}b\textsf{,}c\}, labelstyle={xshift=-105pt, yshift=20pt}, style={in = 270, out = 270}](x_3)(y'_1) 
\Edge[label = \{c\}, labelstyle={xshift=-68pt, yshift=13pt}, style={in = 270, out = 270,min distance=2cm}](x_4)(y'_1) 
\Edge[label = \{a\}, labelstyle={xshift=-40pt, yshift=-8pt}, style={in = 270, out = 270,min distance=1cm}](x_5)(y'_1) 

\def\x{-4}
\def\y{-6.0}	
\Edge[label = \{b\}, labelstyle={xshift=-85pt, yshift=180pt}, style={in = 00, out = 0,min distance=2cm}](y_0)(y'_0) 
\Edge[label = \{b\textsf{,}c\}, labelstyle={xshift=-70pt, yshift=160pt}, style={in = 0, out = 0,min distance=2cm}](y_1)(y'_0)
\Edge[label = \{b\}, labelstyle={xshift=-50pt, yshift=130pt}, style={in = 0, out = 0}](y_2)(y'_0) 
\Edge[label = \{a\textsf{,}b\}, labelstyle={xshift=-35pt, yshift=100pt}, style={in = 0, out = 0}](y_3)(y'_0) 
\Edge[label = \{a\textsf{,}b\textsf{,}c\}, labelstyle={xshift=-40pt, yshift=67pt}, style={in = 0, out = 0,min distance=2cm}](y_4)(y'_0) 
\Edge[label = \{a\textsf{,}c\}, labelstyle={xshift=-20pt, yshift=45pt}, style={in = 0, out = 0,min distance=1cm}](y_5)(y'_0) 

  \def\x{0}
  \def\y{-1.0}
\draw[circle, -,dotted, very thick,rounded corners=8pt] (\x+0.2,\y+2)--(\x+6.4,\y+2) --(\x+6.9,\y+1.5) -- (\x+6.9,\y-7.5)-- (\x+6.4,\y-8.0) -- (\x-0.3,\y-8.0) -- (\x-0.8,\y-7.5) -- (\x-0.8,\y+1.5) -- (\x-0.3,\y+2)--(\x+0.2,\y+2);
\draw[circle, -,dotted, very thick,rounded corners=8pt] (\x+0.0,\y-8)  -- (\x-1,\y-9.0)--(\x-1.5,\y-9.0);
\node at (\x-1.8,\y-9) {$X_2$};
  \def\x{1.9}
  \def\y{-5.0}
\draw[circle, -, very thick,rounded corners=8pt] (\x+0.4,\y+2)  -- (\x+6,\y-0.0)--(\x+6.5,\y-0.0);
\draw[circle, -, very thick,rounded corners=8pt] (\x+0.4,\y-2)  -- (\x+6,\y-0.0)--(\x+6.5,\y-0.0);
\draw[circle, -, very thick,rounded corners=8pt] (\x+5.4,\y+2)  -- (\x+6,\y-0.0)--(\x+6.5,\y-0.0);
\draw[circle, -, very thick,rounded corners=8pt] (\x+5.4,\y-2)  -- (\x+6,\y-0.0)--(\x+6.5,\y-0.0);
\node at (\x+6.9,\y-0) {$X_1$};
  \def\x{2}
  \def\y{-3.0}
\draw[circle, -, very thick,rounded corners=8pt] (\x+0.2,\y+2)--(\x+0.4,\y+2) --(\x+0.9,\y+1.5) -- (\x+0.9,\y+0.5)-- (\x+0.4,\y-0.0) -- (\x-0.3,\y-0.0) -- (\x-0.8,\y+0.5) -- (\x-0.8,\y+1.5) -- (\x-0.3,\y+2)--(\x+0.2,\y+2);
  \def\x{6}
  \def\y{-3.0}
\draw[circle, -, very thick,rounded corners=8pt] (\x+0.2,\y+2)--(\x+2.4,\y+2) --(\x+2.9,\y+1.5) -- (\x+2.9,\y+0.5)-- (\x+2.4,\y-0.0) -- (\x-0.3,\y-0.0) -- (\x-0.8,\y+0.5) -- (\x-0.8,\y+1.5) -- (\x-0.3,\y+2)--(\x+0.2,\y+2);
  \def\x{2}
  \def\y{-9.0}
\draw[circle, -, very thick,rounded corners=8pt] (\x+0.2,\y+2)--(\x+0.4,\y+2) --(\x+0.9,\y+1.5) -- (\x+0.9,\y-1.5)-- (\x+0.4,\y-2.0) -- (\x-0.3,\y-2.0) -- (\x-0.8,\y-1.5) -- (\x-0.8,\y+1.5) -- (\x-0.3,\y+2)--(\x+0.2,\y+2);
  \def\x{6}
  \def\y{-9.0}
\draw[circle, -, very thick,rounded corners=8pt] (\x+0.2,\y+2)--(\x+2.4,\y+2) --(\x+2.9,\y+1.5) -- (\x+2.9,\y-1.5)-- (\x+2.4,\y-2.0) -- (\x-0.3,\y-2.0) -- (\x-0.8,\y-1.5) -- (\x-0.8,\y+1.5) -- (\x-0.3,\y+2)--(\x+0.2,\y+2);
  \def\x{1.9}
  \def\y{-7.0}
  \draw[circle, -,dashdotted, very thick,rounded corners=8pt] (\x+0.2,\y+2)--(\x+4.5,\y+2) --(\x+5,\y+1.5) -- (\x+5,\y-4)-- (\x+4.5,\y-4.5) -- (\x-0.3,\y-4.5) -- (\x-0.8,\y-4) -- (\x-0.8,\y+1.5) -- (\x-0.3,\y+2)--(\x+0.2,\y+2);
    \draw[circle, -,dashdotted, very thick,rounded corners=8pt] (\x+8.2,\y+2)--(\x+8.5,\y+2) --(\x+9,\y+1.5) -- (\x+9,\y-4)-- (\x+8.5,\y-4.5) -- (\x+8-0.3,\y-4.5) -- (\x+8-0.8,\y-4) -- (\x+8-0.8,\y+1.5) -- (\x+8-0.3,\y+2)--(\x+8.2,\y+2);
\def\x{1.6}
\def\y{-1.0}
    \draw[circle, -,dashdotted, very thick,rounded corners=8pt] (\x-0.0,\y-10.5)  -- (\x-1.5,\y-10.75)--(\x-3.0,\y-10.9);
        \draw[circle, -,dashdotted, very thick,rounded corners=8pt] (\x+8.0,\y-10.5)  -- (\x-0.0,\y-11.1)--(\x-3.0,\y-10.9);

\node at (\x-3.3,\y-11) {$X_3$};
\def\x{2}
\def\y{-7.0}

  \def\x{-0.1+12}
  \def\y{-0.9-12}
  \draw[circle, -,dashed, very thick,rounded corners=8pt] (\x+0.2,\y+2)--(\x+0.6,\y+2) --(\x+1.1,\y+1.5) -- (\x+1.1,\y+0.3)-- (\x+0.6,\y-0.2) -- (\x-0.3,\y-0.2) -- (\x-0.8,\y+0.3) -- (\x-0.8,\y+1.5) -- (\x-0.3,\y+2)--(\x+0.2,\y+2);
    \def\y{-0.9-11.5}
  \draw[circle, -,dashed, very thick,rounded corners=8pt] (\x-0.5,\y-0.5)  -- (\x-1,\y-1.0)--(\x-1.5,\y-1.0);
\node at (\x-1.8,\y-1) {$X_4$};

  \def\x{17}
  \def\y{-5}
  \node at (\x-1.8,\y-11) {$G/_{i=1}^3X_{1,i}/_{i=1}^4X_{2,i}/_{i=1}^4X_{3,i}/X_{4,1}$};
    \def\x{17}
  \def\y{12.5}
  \node at (\x-1.8,\y-11.25) {$G/_{i=1}^3X_{1,i}/_{i=1}^5X_{2,i}/_{i=1}^3X_{3,i}/X_{4,1}$};
\end{tikzpicture}
}
\end{center}
\caption{The decomposition of the graph $G$ into the graphs $G/_{i=1}^3X'_{1,i}/_{i=1}^4X'_{2,i}$ $/_{i=1}^4X'_{3,i}/X'_{4,1}$ and $G/_{i=1}^3X''_{1,i}/_{i=1}^5X''_{2,i}/_{i=1}^3X''_{3,i}/X''_{4,1}$, with $G\cong G/_{i=1}^3X'_{1,i}/_{i=1}^4X'_{2,i}/_{i=1}^4X'_{3,i}$ $/X'_{4,1}\boxbackslash G/_{i=1}^3X''_{1,i}/_{i=1}^5X''_{2,i}/_{i=1}^3X''_{3,i}/X''_{4,1}$.}
  \label{BiPartiteExampleDecomposition1}
\end{figure}
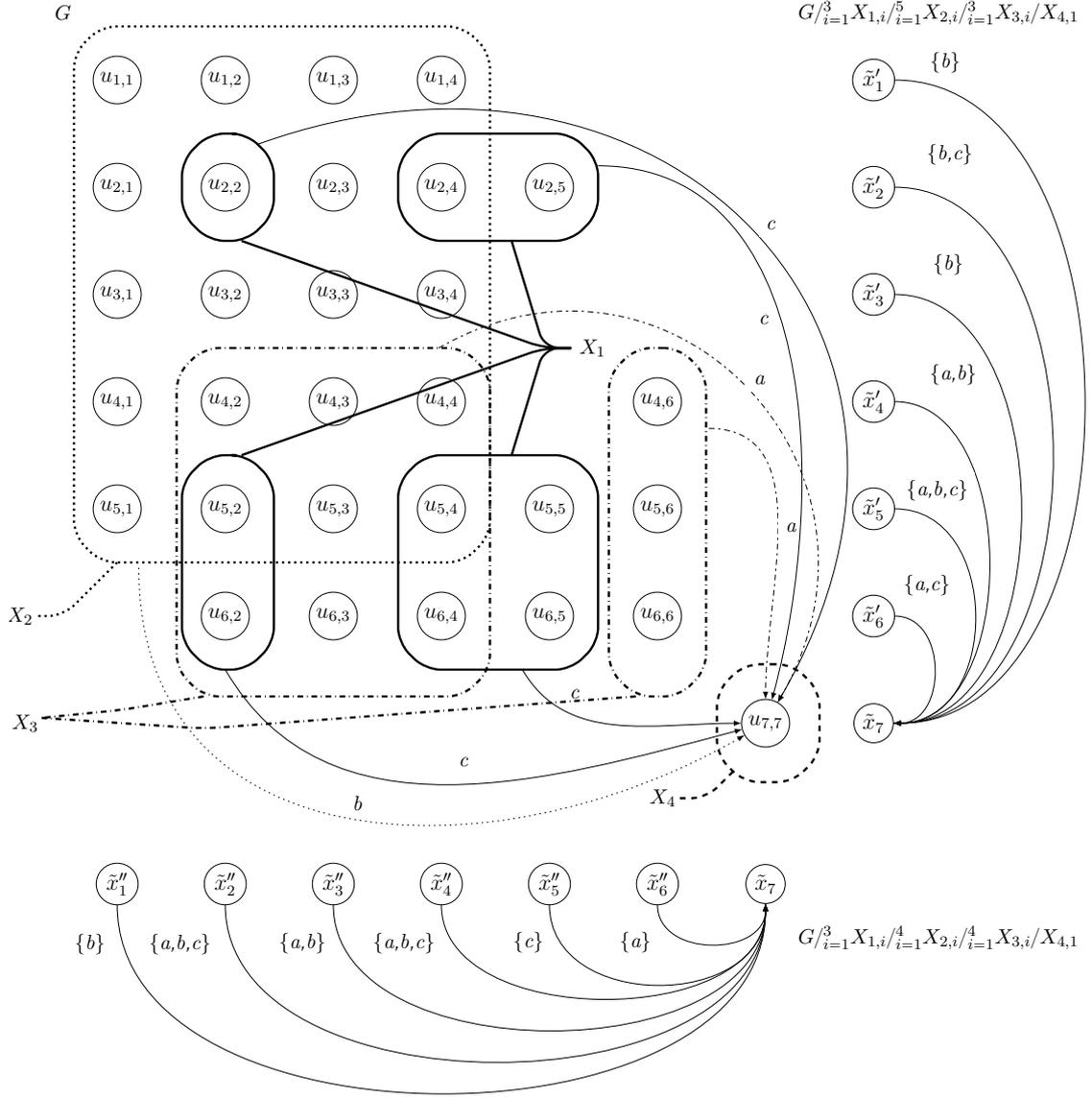

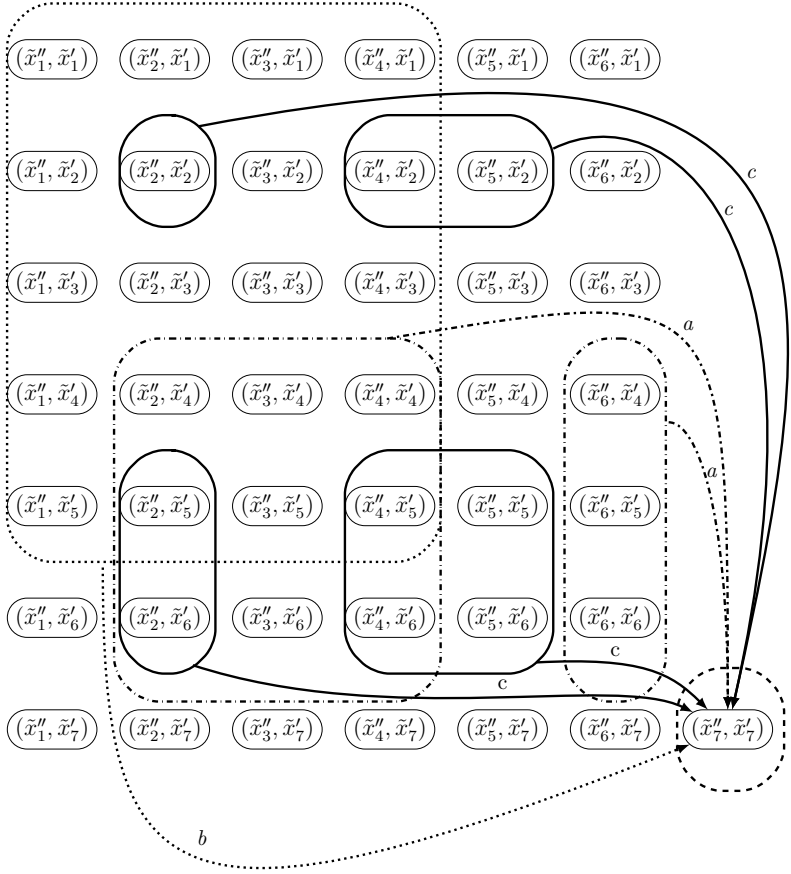
\begin{figure}[H]
\begin{center}
\resizebox{1\textwidth}{!}{
\begin{tikzpicture}[->,>=latex,shorten >=0pt,auto,node distance=2.5cm,
  main node/.style={circle,fill=blue!10,draw, font=\sffamily\Large\bfseries}]
  \tikzset{VertexStyle/.append style={
  font=\itshape\large, shape = circle,inner sep = 2pt, outer sep = 0pt,minimum size = 20 pt,draw}}
  \tikzset{EdgeStyle/.append style={thin}}
  \tikzset{LabelStyle/.append style={font = \itshape}}
  \SetVertexMath

  \clip (-2,2.5) rectangle (18, -15.0);
  \tikzset{VertexStyle/.append style={
  font=\itshape\large, shape = rounded rectangle, inner sep = 2pt, outer sep = 0pt,minimum size = 20 pt,draw}}
  \def\x{5.0}
  \def\y{2.5}
\node at (\x-1.0,\y-0.75) {$G/_{i=1}^3X_{1,i}/_{i=1}^4X_{2,i}/_{i=1}^4X_{3,i}/X_{4,1}\boxtimes G/_{i=1}^3X_{1,i}/_{i=1}^5X_{2,i}/_{i=1}^3X_{3,i}/X_{4,1}$};
  \def\x{0.0}
  \def\y{0.0}
  \Vertex[x=\x+0, y=\y+0,L={(\tilde{x}_1'',\tilde{x}_1')}]{u_00}
  \Vertex[x=\x+2, y=\y+0,L={(\tilde{x}_2'',\tilde{x}_1')}]{u_01}
  \Vertex[x=\x+4, y=\y+0,L={(\tilde{x}_3'',\tilde{x}_1')}]{u_02}
  \Vertex[x=\x+6, y=\y+0,L={(\tilde{x}_4'',\tilde{x}_1')}]{u_03}
  \Vertex[x=\x+8, y=\y+0,L={(\tilde{x}_5'',\tilde{x}_1')}]{u_04}
  \Vertex[x=\x+10, y=\y+0,L={(\tilde{x}_6'',\tilde{x}_1')}]{u_05}
  \def\y{-2.0}
  \Vertex[x=\x+0, y=\y+0,L={(\tilde{x}_1'',\tilde{x}_2')}]{u_10}
  \Vertex[x=\x+2, y=\y+0,L={(\tilde{x}_2'',\tilde{x}_2')}]{u_11}
  \Vertex[x=\x+4, y=\y+0,L={(\tilde{x}_3'',\tilde{x}_2')}]{u_12}
  \Vertex[x=\x+6, y=\y+0,L={(\tilde{x}_4'',\tilde{x}_2')}]{u_13}
  \Vertex[x=\x+8, y=\y+0,L={(\tilde{x}_5'',\tilde{x}_2')}]{u_14}
  \Vertex[x=\x+10, y=\y+0,L={(\tilde{x}_6'',\tilde{x}_2')}]{u_15}
  \def\y{-4.0}
  \Vertex[x=\x+0, y=\y+0,L={(\tilde{x}_1'',\tilde{x}_3')}]{u_20}
  \Vertex[x=\x+2, y=\y+0,L={(\tilde{x}_2'',\tilde{x}_3')}]{u_21}
  \Vertex[x=\x+4, y=\y+0,L={(\tilde{x}_3'',\tilde{x}_3')}]{u_22}
  \Vertex[x=\x+6, y=\y+0,L={(\tilde{x}_4'',\tilde{x}_3')}]{u_23}
  \Vertex[x=\x+8, y=\y+0,L={(\tilde{x}_5'',\tilde{x}_3')}]{u_24}
  \Vertex[x=\x+10, y=\y+0,L={(\tilde{x}_6'',\tilde{x}_3')}]{u_25}
  \def\y{-6.0}
  \Vertex[x=\x+0, y=\y+0,L={(\tilde{x}_1'',\tilde{x}_4')}]{u_30}
  \Vertex[x=\x+2, y=\y+0,L={(\tilde{x}_2'',\tilde{x}_4')}]{u_31}
  \Vertex[x=\x+4, y=\y+0,L={(\tilde{x}_3'',\tilde{x}_4')}]{u_32}
  \Vertex[x=\x+6, y=\y+0,L={(\tilde{x}_4'',\tilde{x}_4')}]{u_33}
  \Vertex[x=\x+8, y=\y+0,L={(\tilde{x}_5'',\tilde{x}_4')}]{u_34}
  \Vertex[x=\x+10, y=\y+0,L={(\tilde{x}_6'',\tilde{x}_4')}]{u_35}
  \def\y{-8.0}
  \Vertex[x=\x+0, y=\y+0,L={(\tilde{x}_1'',\tilde{x}_5')}]{u_40}
  \Vertex[x=\x+2, y=\y+0,L={(\tilde{x}_2'',\tilde{x}_5')}]{u_41}
  \Vertex[x=\x+4, y=\y+0,L={(\tilde{x}_3'',\tilde{x}_5')}]{u_42}
  \Vertex[x=\x+6, y=\y+0,L={(\tilde{x}_4'',\tilde{x}_5')}]{u_43}
  \Vertex[x=\x+8, y=\y+0,L={(\tilde{x}_5'',\tilde{x}_5')}]{u_44}
  \Vertex[x=\x+10, y=\y+0,L={(\tilde{x}_6'',\tilde{x}_5')}]{u_45}
  \def\y{-10.0}
  \Vertex[x=\x+0, y=\y+0,L={(\tilde{x}_1'',\tilde{x}_6')}]{u_50}
  \Vertex[x=\x+2, y=\y+0,L={(\tilde{x}_2'',\tilde{x}_6')}]{u_51}
  \Vertex[x=\x+4, y=\y+0,L={(\tilde{x}_3'',\tilde{x}_6')}]{u_52}
  \Vertex[x=\x+6, y=\y+0,L={(\tilde{x}_4'',\tilde{x}_6')}]{u_53}
  \Vertex[x=\x+8, y=\y+0,L={(\tilde{x}_5'',\tilde{x}_6')}]{u_54}
  \Vertex[x=\x+10, y=\y+0,L={(\tilde{x}_6'',\tilde{x}_6')}]{u_55}

  \def\y{-12.0}
  \Vertex[x=\x+0, y=\y+0,L={(\tilde{x}_1'',\tilde{x}_7')}]{u_60}
  \Vertex[x=\x+2, y=\y+0,L={(\tilde{x}_2'',\tilde{x}_7')}]{u_61}
  \Vertex[x=\x+4, y=\y+0,L={(\tilde{x}_3'',\tilde{x}_7')}]{u_62}
  \Vertex[x=\x+6, y=\y+0,L={(\tilde{x}_4'',\tilde{x}_7')}]{u_63}
  \Vertex[x=\x+8, y=\y+0,L={(\tilde{x}_5'',\tilde{x}_7')}]{u_64}
  \Vertex[x=\x+10, y=\y+0,L={(\tilde{x}_6'',\tilde{x}_7')}]{u_65}
  \Vertex[x=\x+12, y=\y+0,L={(\tilde{x}_7'',\tilde{x}_7')}]{v_0}

\def\x{-4}
\def\y{-6.0}	

\draw (\x+5.2+4.75,\y+1.0)[very thick, dashdotted,font={\itshape}] .. controls ($ (\x+16,\y+1.0) +(0,1)$) and ($ (v_0) +(0,8)$) .. (v_0) node[pos=0.5, inner sep=-1pt,  label={a}] {};
\draw (\x+13.2+1.75,\y-0.5)[very thick, dashdotted,font={\itshape}] .. controls ($ (\x+13.2+1.75,\y-0.5) +(0.5,0)$) and ($ (v_0) +(0,4)$) .. (v_0) node[pos=0.5, inner sep=-1pt,  label={a}] {};
\draw (\x+4.9,\y-3.1)[very thick, dotted,font={\itshape}] .. controls ($ (\x+4.9,\y-3.0) +(0,-6)$) and ($ (v_0) +(-10,-4)$) .. (v_0) node[pos=0.5, inner sep=-1pt,  label={b}] {};
\draw (\x+5.5+1.1,\y+4.8)[very thick, font={\itshape}] .. controls ($ (\x+5.5+0.6,\y+5.0) +(12,2)$) and ($ (v_0) +(2,9)$) .. (v_0) node[pos=0.5, inner sep=-1pt,  label={c}] {};
\draw (\x+9.5+3.4,\y+4.8-0.4)[very thick, font={\itshape}] .. controls ($ (\x+9.5+3.4,\y+4.8-0.4) +(2,1)$) and ($ (v_0) +(2,9)$) .. (v_0) node[pos=0.5, inner sep=-1pt,  label={c}] {};
\draw (\x+5.5+1,\y+4.8-9.6-0.05)[very thick, font={\itshape}] .. controls ($ (\x+5.5+0.5,\y+4.8-9.6-0.2) +(4,-1)$) and ($ (v_0) +(-2,1)$) .. (v_0){};
\node at (\x+12, \y-5.2) {c};


\draw (\x+9.5+3.1,\y+4.8-9.6)[very thick, font={\itshape}] .. controls ($(\x+9.5+5,\y+4.8-9.5)$) and ($(v_0) +(-1,1)$) .. (v_0){};
\node at (\x+14, \y-4.6) {c};

  \def\x{0}
  \def\y{-1.0}
\draw[circle, -,dotted, very thick,rounded corners=8pt] (\x+0.2,\y+2)--(\x+6.4,\y+2) --(\x+6.9,\y+1.5) -- (\x+6.9,\y-7.5)-- (\x+6.4,\y-8.0) -- (\x-0.3,\y-8.0) -- (\x-0.8,\y-7.5) -- (\x-0.8,\y+1.5) -- (\x-0.3,\y+2)--(\x+0.2,\y+2);
  \def\x{1.9}
  \def\y{-5.0}
  \def\x{2}
  \def\y{-3.0}
\draw[circle, -, very thick,rounded corners=8pt] (\x+0.2,\y+2)--(\x+0.4,\y+2) --(\x+0.9,\y+1.5) -- (\x+0.9,\y+0.5)-- (\x+0.4,\y-0.0) -- (\x-0.3,\y-0.0) -- (\x-0.8,\y+0.5) -- (\x-0.8,\y+1.5) -- (\x-0.3,\y+2)--(\x+0.2,\y+2);
  \def\x{6}
  \def\y{-3.0}
\draw[circle, -, very thick,rounded corners=8pt] (\x+0.2,\y+2)--(\x+2.4,\y+2) --(\x+2.9,\y+1.5) -- (\x+2.9,\y+0.5)-- (\x+2.4,\y-0.0) -- (\x-0.3,\y-0.0) -- (\x-0.8,\y+0.5) -- (\x-0.8,\y+1.5) -- (\x-0.3,\y+2)--(\x+0.2,\y+2);
  \def\x{2}
  \def\y{-9.0}
\draw[circle, -, very thick,rounded corners=8pt] (\x+0.2,\y+2)--(\x+0.4,\y+2) --(\x+0.9,\y+1.5) -- (\x+0.9,\y-1.5)-- (\x+0.4,\y-2.0) -- (\x-0.3,\y-2.0) -- (\x-0.8,\y-1.5) -- (\x-0.8,\y+1.5) -- (\x-0.3,\y+2)--(\x+0.2,\y+2);
  \def\x{6}
  \def\y{-9.0}
\draw[circle, -, very thick,rounded corners=8pt] (\x+0.2,\y+2)--(\x+2.4,\y+2) --(\x+2.9,\y+1.5) -- (\x+2.9,\y-1.5)-- (\x+2.4,\y-2.0) -- (\x-0.3,\y-2.0) -- (\x-0.8,\y-1.5) -- (\x-0.8,\y+1.5) -- (\x-0.3,\y+2)--(\x+0.2,\y+2);
  \def\x{1.9}
  \def\y{-7.0}
  \draw[circle, -,dashdotted, very thick,rounded corners=8pt] (\x+0.2,\y+2)--(\x+4.5,\y+2) --(\x+5,\y+1.5) -- (\x+5,\y-4)-- (\x+4.5,\y-4.5) -- (\x-0.3,\y-4.5) -- (\x-0.8,\y-4) -- (\x-0.8,\y+1.5) -- (\x-0.3,\y+2)--(\x+0.2,\y+2);
    \draw[circle, -,dashdotted, very thick,rounded corners=8pt] (\x+8.2,\y+2)--(\x+8.5,\y+2) --(\x+9,\y+1.5) -- (\x+9,\y-4)-- (\x+8.5,\y-4.5) -- (\x+8-0.3,\y-4.5) -- (\x+8-0.8,\y-4) -- (\x+8-0.8,\y+1.5) -- (\x+8-0.3,\y+2)--(\x+8.2,\y+2);
\def\x{1.6}
\def\y{-1.0}

\def\x{2}
\def\y{-7.0}

  \def\x{-0.1+12}
  \def\y{-0.9-12}
  \draw[circle, -,dashed, very thick,rounded corners=8pt] (\x+0.2,\y+2)--(\x+0.6,\y+2) --(\x+1.1,\y+1.5) -- (\x+1.1,\y+0.3)-- (\x+0.6,\y-0.2) -- (\x-0.3,\y-0.2) -- (\x-0.8,\y+0.3) -- (\x-0.8,\y+1.5) -- (\x-0.3,\y+2)--(\x+0.2,\y+2);
    \def\y{-0.9-11.5}

  \def\x{2}
  \def\y{2}

\end{tikzpicture}
}
\end{center}
\caption{The intermediate stage of $G/_{i=1}^3X'_{1,i}/_{i=1}^4X'_{2,i}/_{i=1}^4X'_{3,i}/X'_{4,1}$ and $G/_{i=1}^3X''_{1,i}$ $/_{i=1}^5X''_{2,i}/_{i=1}^3X''_{3,i}/X''_{4,1}$, $G/_{i=1}^3X'_{1,i}/_{i=1}^4X'_{2,i}/_{i=1}^4X'_{3,i}/X'_{4,1}\boxtimes G/_{i=1}^3X''_{1,i}/_{i=1}^5X''_{2,i}/_{i=1}^3X''_{3,i}/X''_{4,1}$.}
  \label{BiPartiteExampleDecomposition2}
\end{figure}
\begin{theorem}\label{theorem_5}
Let $G$ be a bipartite matrix graph consisting of semicomplete bipartite subgraphs $B(X_a,X_b)$ only, where each $B(X_a,X_b)$ is arc-induced by a set of all arcs of $G$ with identical labels, $V(G)=X_1\cup\ldots \cup X_x$
, $a,b\in\{1,\ldots,x\}, a\neq b$.
Let $[X_a,X_b]$ have only forward arcs or let $[X_a,X_b]$ have only backward arcs.
Let there be no arc $a=u_iv_j$ in $G$ with $u_i,v_j\in X_a$ or $u_i,v_j\in X_b$.
Let $X_a=\{v_{i,j}\mid i\in I_{X_a}\subseteq I=\{1,\ldots,m\},j\in J_{X_a}\subseteq J=\{1,\ldots,n\}\}$, $|X_a|=k_a\cdot l_a$, $k_a, l_a\in \mathbb{N}^+,|I_{X_a}|=k_a,|J_{X_a}|=l_a$, with rows $X'_{a,i}=\{v_{i,j}\mid j\in J_{X_a}\},i\in I_{X_a}$ and columns $X''_{a,j}=\{v_{i,j}\mid i\in I_{X_a}\},j\in J_{X_a}$ and let $X_b=\{v_{i,j}\mid i\in I_{X_b}\subseteq I=\{1,\ldots,m\},j\in J_{X_b}\subseteq J=\{1,\ldots,n\}\}$, $|X_b|=k_b\cdot l_b$, $k_b, l_b\in \mathbb{N}^+,|I_{X_b}|=k_b,|J_{X_b}|=l_b$, with rows $X'_{b,i}=\{v_{i,j}\mid j\in J_{X_b}\},i\in I_{X_b}$ and columns $X''_{b,j}=\{v_{i,j}\mid i\in I_{X_b}\},j\in J_{X_b}$. 
If the intersection of $X_i$ and $X_j$ is empty or the intersection of $X_i$ and $X_j$ is a grid, for any $X_i$ and any $X_j$ of $G$ for $i,j\in\{1,\ldots,x\}$ then $G\cong G/_{y=1}^{x}/_{z=1}^{k_y}X'_{y,z}\boxbackslash G/_{y=1}^{x}/_{z=1}^{l_y}X''_{y,z}$.
\end{theorem}
\begin{proof}
It suffices to define a mapping $\phi: V(G)\rightarrow V(G/_{y=1}^{x}/_{z=1}^{k_y}X'_{y,z}\boxbackslash G/_{y=1}^{x}$ $/_{z=1}^{l_y}X''_{y,z})$ and to prove that $\phi$ is an isomorphism from $G$ to $G/_{y=1}^{x}/_{z=1}^{k_y}X'_{y,z}$ $/_{y=1}^{x}\boxbackslash G/_{y=1}^{x}/_{z=1}^{l_y}X''_{y,z}$.
Let $\tilde{x}'_i$ be the new vertex replacing the set $X'_{y,z}$ with $v_{i,j}\in X'_{y,z}$, $\tilde{x}''_j$  be the new vertex replacing the set $X''_{y,z}$ with $v_{i,j}\in X''_{y,z}$, when defining $G/_{y=1}^{x}/_{z=1}^{k_y}X'_{y,z}$ and $G/_{y=1}^{x}/_{z=1}^{l_y}X''_{y,z}$, respectively.
Let $\tilde{x}'_i$ be the new vertex replacing the vertices $v_{i,j}\in X'_{y,z}$, $\tilde{x}''_j$  be the new vertex replacing the vertices $v_{i,j}\in X''_{y,z}$, when defining $G/_{y=1}^{x}/_{z=1}^{k_y}X'_{y,z}$ and $G/_{y=1}^{x}/_{z=1}^{l_y}X''_{y,z}$, respectively.
Consider the mapping $\phi: V(G)\rightarrow V(G/_{y=1}^{x}/_{z=1}^{k_y}X'_{y,z}\boxbackslash G/_{y=1}^{x}/_{z=1}^{l_y}$ $X''_{y,z})$ defined by $\phi(v_{i,j})=(\tilde{x}'_i,\tilde{x}''_j)$. 
Then $\phi$ is obviously a bijection if  $V(G/_{y=1}^{x}/_{z=1}^{k_y}X'_{y,z}$ $\boxbackslash G/_{y=1}^{x}/_{z=1}^{l_y}X''_{y,z})=Z$, where $Z$ is defined as $Z=\{(\tilde{x}'_i,\tilde{x}''_j) \mid v_{i,j}\in V(G)$, $ \phi(v_{i,j})=(\tilde{x}'_i,\tilde{x}''_j)\}$. 
We are going to show this later by arguing that all vertices $\tilde{x}'_i$ and $\tilde{x}'_j$, $i\neq j$, are different, and that all vertices $\tilde{x}''_i$ and $\tilde{x}''_j$, $i\neq j$, are different and that all the other vertices $(\tilde{x}'_k,\tilde{x}''_l)$ of $G/_{y=1}^{x}/_{z=1}^{k_y}X'_{y,z}$ $\Box G/_{y=1}^{x}/_{z=1}^{l_y}X''_{y,z}$ for which there is no $v_{k,l}\in V(G)$ will disappear from $G/_{y=1}^{x}/_{z=1}^{k_y}X'_{y,z}$ $\boxbackslash G/_{y=1}^{x}/_{z=1}^{l_y}X''_{y,z}$. 	
But first we are going to prove the following claim. 
\begin{claim}\label{claim3}
The subgraph of $G/_{y=1}^{x}/_{z=1}^{k_y}X'_{y,z}$ $\boxtimes G/_{y=1}^{x}/_{z=1}^{l_y}X''_{y,z}$ induced by $Z$ is isomorphic to $G$.
\end{claim}
\begin{proof}
We start with proving that $\tilde{x}'_i$ and $\tilde{x}'_j, i\neq j,$ implies $\tilde{x}'_i\neq \tilde{x}'_j$ and that $\tilde{x}''_i$ and $\tilde{x}''_j, i\neq j,$ implies $\tilde{x}''_i\neq \tilde{x}'_j$.
Let $R_i$ be the set of rows with all vertices $v_{i,j}$ of $V(G)$.
Therefore, all rows in $R_i$ have the number $i$ as their first index.
Because $G$ is a bipartite matrix graph, we have that for any division of $R_i$ into the sets $R_{i_1}$ and $R_{i_2}$, $R_i=R_{i_1}\cup R_{i_2}$, there is always a row $X'_{k,x}\in R_{i_1}$ and a row $X'_{l,y}\in R_{i_2}$ with $X'_{k,x}\cap X'_{l,y}\neq \emptyset$.
Therefore, all rows of $R_i$ are contracted to $\tilde{x}'_i$. 
Because all rows with vertices $v_{i,j}$ are in $R_i$, a row with a vertex $v_{k,l}$ with $v_{k,l}$ not in in any row of $R_i$ must have $i\neq k$.
Likewise, let $R_j$ be the set of columns with all vertices $v_{i,j}$ of $V(G)$.
Therefore, all columns in $R_j$ have the number $j$ as their second index.
Because $G$ is a bipartite matrix graph, we have that for any division of $R_j$ into the sets $R_{j_1}$ and $R_{j_2}$, $R_j=R_{j_1}\cup R_{j_2}$, there is always a column $X''_{k,x}\in R_{j_1}$ and a column $X''_{l,y}\in R_{j_2}$ with $X''_{k,x}\cap X''_{l,y}\neq \emptyset$.
Therefore, all columns of $R_j$ are contracted to $\tilde{x}''_i$. 
Because all columns with vertices $v_{i,j}$ are in $R_j$, a column with a vertex $v_{k,l}$ with $v_{k,l}$ not in in any column of $R_j$ must have $j\neq l$.
Hence, we have that $\tilde{x}'_i$ and $\tilde{x}'_j, i\neq j,$ implies $\tilde{x}'_i\neq \tilde{x}'_j$ and that $\tilde{x}''_i$ and $\tilde{x}''_j, i\neq j,$ implies $\tilde{x}''_i\neq \tilde{x}'_j$.
Therefore, $\phi$ maps each vertex $v_{i,j}\in V(G)$ to $(\tilde{x}'_i,\tilde{x}''_j)\in Z$ and if $v_{i_1,j_1}\neq v_{i_2,j_2}$ then  $(\tilde{x}'_{i_1},\tilde{x}'_{i_2})\neq (\tilde{x}''_{j_1}\tilde{x}''_{j_2})$, $v_{i_1,j_1}, v_{i_2,j_2}\in V(G)$ and we have that $\phi$ a bijection from $V(G)$ to $Z$. 
It remains to show that this bijection preserves the arcs and their labels.

Because there is no arc $a=v_{i,j}v_{k.l}$ in $A(G)$ with $v_{i,j},v_{k,l}\in X_a$ or $v_{i,j},v_{k,l}\in X_b$, we have that by the contractions each arc $a\in A(G)$ with $\mu(a)=(v_{i,j},v_{k,l}), \lambda(a)=a'$ is replaced by an arc $x'\in G/_{y=1}^{x}/_{z=1}^{k_y}X'_{y,z}$ with $\mu(x')=(\tilde{x}'_i,\tilde{x}'_k), \lambda(x')=a'$ and an arc $x''\in G/_{y=1}^{x}/_{z=1}^{l_y}X''_{y,z}$ with $\mu(x'')=(\tilde{x}''_j,\tilde{x}''_l), \lambda(x'')=a'$.
Therefore, all arcs $x'$ are synchronising arcs of $G/_{y=1}^{x}/_{z=1}^{k_y}X'_{y,z}$ with respect to $G/_{y=1}^{x}/_{z=1}^{l_y}X''_{y,z}$ (by hypothesis) and all arcs $x''$ are synchronising arcs of $G/_{y=1}^{x}/_{z=1}^{l_y}X''_{y,z}$ with respect to $G/_{y=1}^{x}/_{z=1}^{k_y}X'_{y,z}$ (by hypothesis).  
It follows that the arcs $x'$ and $x''$ correspond to an arc $y=(\tilde{x}'_i,\tilde{x}''_j)(\tilde{x}'_{k},\tilde{x}''_{l})$ of $G/_{y=1}^{x}/_{z=1}^{k_y}X'_{y,z}$ $\boxtimes G/_{y=1}^{x}/_{z=1}^{l_y}X''_{y,z}$ with $\lambda(y)=\lambda(x')$.
Furthermore, $\phi$ maps vertices $v_{i,j}$ and $v_{k,l}$ on vertices $(\tilde{x}'_i,\tilde{x}''_j)$ and $(\tilde{x}'_k,\tilde{x}''_l)$, respectively, and therefore we have that an arc $z=v_{i,j}v_{k,l}$ of $G$ corresponds with an arc $y=(\tilde{x}'_i,\tilde{x}''_j)(\tilde{x}'_k,\tilde{x}''_l)$ of $G/_{y=1}^{x}/_{z=1}^{k_y}X'_{y,z}$ $\boxtimes G/_{y=1}^{x}/_{z=1}^{l_y}X''_{y,z}$, with $\lambda(y)=\lambda(z)$. 
Because $(\tilde{x}'_i,\tilde{x}''_j)$ and $(\tilde{x}'_{k},\tilde{x}''_{l})$ are in $Z$, the arc $y$ is an arc of the graph induced by $Z$ and we have the one-to-one relationship between the arcs $y$ of $G/_{y=1}^{x}/_{z=1}^{k_y}X'_{y,z}$ $\boxtimes G/_{y=1}^{x}/_{z=1}^{l_y}X''_{y,z}$ and $z$ in $G$.
Together with, there are no other vertices in $Z$ than $(\tilde{x}'_i,\tilde{x}''_j)$ and $(\tilde{x}'_{k},\tilde{x}''_{l}))$ and there are no other vertices in $G$ than $v_{i,j}$ and $v_{k,l}$, the subgraph of $G/_{y=1}^{x}/_{z=1}^{k_y}X'_{y,z}$ $\boxtimes G/_{y=1}^{x}/_{z=1}^{l_y}X''_{y,z}$ induced by $Z$ is isomorphic to $G$.
\end{proof}

By the definition of the Cartesian product, for each pair of vertices $\tilde{x}'_{i} \in V(G/_{y=1}^{x}/_{z=1}^{k_y}$ $X'_{y,z})$ and $\tilde{x}''_{j} \in V(G/_{y=1}^{x}/_{z=1}^{l_y}X''_{y,z})$, there exists a vertex $(\tilde{x}'_{i},\tilde{x}''_{j}) \in V(G/_{y=1}^{x}/_{z=1}^{k_y}X'_{y,z} \boxtimes G/_{y=1}^{x}/_{z=1}^{l_y}X''_{y,z})$. 
It remains to show that $\phi$ is a bijection from $V(G)$ to $Z'=V(G/_{y=1}^{x}/_{z=1}^{k_y}$ $X'_{y,z}\boxbackslash G/_{y=1}^{x}/_{z=1}^{l_y}X''_{y,z})$ preserving the arcs and their labels.
Therefore, we have to show that all vertices of $V(G/_{y=1}^{x}/_{z=1}^{k_y}X'_{y,z}\boxtimes G/_{y=1}^{x}/_{z=1}^{l_y}X''_{y,z})$ not in $Z$ are removed from $V(G/_{y=1}^{x}/_{z=1}^{k_y}$ $X'_{y,z}\boxtimes G/_{y=1}^{x}/_{z=1}^{l_y}X''_{y,z})$.
Let $|G/_{y=1}^{x}/_{z=1}^{k_y}$ $/X'_{y,z}|=m_1\leq m$ and $|G/_{y=1}^{x}/_{z=1}^{l_y}X''_{y,z}|=n_1\leq n$.
Let $v_{s,t}\notin G$ with $s\in \{1,\ldots ,m_1\}$ and $t\in \{1,\ldots ,n_1\}$.
Then there cannot exist an arc $v_{i,j}v_{s,t}\in A(G)$ otherwise $v_{s,t}$ must be in $V(G)$. 
But there exist a vertex $\tilde{x}'_s\in G/_{y=1}^{x}/_{z=1}^{k_y}$ $X'_{y,z}$ and a vertex $\tilde{x}''_t \in G/_{y=1}^{x}/_{z=1}^{l_y}X''_{y,z}$, and, therefore, there exists a vertex $(\tilde{x}'_s,\tilde{x}''_t)\in V(G/_{y=1}^{x}/_{z=1}^{k_y}$ $/X'_{y,z}\boxtimes G/_{y=1}^{x}/_{z=1}^{l_y}X''_{y,z})$. 
The intersection of the set of labels $L'$ of arcs with head $\tilde{x}'_s$ and the set of labels $L''$ of arcs with head $\tilde{x}''_t$ is empty, because otherwise there exists an arc $a$ in $A(G)$ with head $v_{s,t}$. 
Hence, all arcs with head $\tilde{x}'_s$ are asynchronous with respect to all arcs with head $\tilde{x}''_t$.
Therefore, there cannot exist a vertex $\tilde{x}'_s,\tilde{x}''_t \in V(G/_{y=1}^{x}/_{z=1}^{k_y}X'_{y,z}\boxbackslash G/_{y=1}^{x}/_{z=1}^{l_y}X''_{y,z}$ and $Z$ must be equal to $V(G/_{y=1}^{x}/_{z=1}^{k_y}X'_{y,z}\boxbackslash G/_{y=1}^{x}/_{z=1}^{l_y}X''_{y,z})$.
Because the subgraph of $V(G/_{y=1}^{x}/_{z=1}^{k_y}X'_{y,z}\boxtimes G/_{y=1}^{x}/_{z=1}^{l_y}X''_{y,z})$ induced by $Z$ is isomorphic to $G$ and $Z=V(G/_{y=1}^{x}/_{z=1}^{k_y}X'_{y,z}\boxbackslash G/_{y=1}^{x}/_{z=1}^{l_y}X''_{y,z})$, it follows that $G \cong G/_{y=1}^{x}/_{z=1}^{k_y}$ $X'_{y,z}\boxbackslash G/_{y=1}^{x}/_{z=1}^{l_y}X''_{y,z}$.
This completes the proof of Theorem~\ref{theorem_5}.
\end{proof}
We call a bipartite matrix graph consisting of semicomplete bipartite subgraphs that is decomposable by Theorem~\ref{theorem_5} a \emph{VRSP-decomposable bipartite matrix graph}.

In the fourth decomposition theorem we are going to prove that $G/_{i\in I_R}R_i\boxbackslash G/_{j\in J_C}C_j\cong G$, where $V(G)$ consists of nonempty pairwise disjoint subsets $R_i=\{v_{i,j}\mid j\in J_C\subseteq J\},i\in I_R\subseteq I$, and nonempty pairwise disjoint subsets $C_j=\{v_{i,j}\mid i\in I_R\subseteq I\},j\in J_C\subseteq J,|I_R|=m_1,|J_C|=n_1,$, with $V(G)=\bigcup\limits_{i\in I_R}R_i=\bigcup\limits_{j\in J_C}C_j$, for which $G[R_x]\cong G[R_y],x,y\in I_R$, $G[C_x]\cong G[C_y],x,y\in J_C$, the arcs of $A_R=\bigcup\limits_{x\in I_R}A[R_x]$ and the arcs of $A_C=\bigcup\limits_{y\in I_C}A[C_y]$ have no labels in common and there are no other arcs in $A(G)$ than the arcs of $A_R$ and the arcs of $A_C$.
We give an illustrative example of the decomposition by Theorem~\ref{theorem_6} in Figure~\ref{ThirdDecomposition}.
\begin{figure}[H]
\begin{center}
\resizebox{0.75\textwidth}{!}{
\begin{tikzpicture}[->,>=latex,shorten >=0pt,auto,node distance=2.5cm,
  main node/.style={circle,fill=blue!10,draw, font=\sffamily\Large\bfseries}]
  \tikzset{VertexStyle/.append style={
  font=\itshape\large, shape = circle,inner sep = 2pt, outer sep = 0pt,minimum size = 20 pt,draw}}
  \tikzset{EdgeStyle/.append style={thin}}
  \tikzset{LabelStyle/.append style={font = \itshape}}
  \SetVertexMath
  \def\x{0.0}
  \def\y{1.0}
\node at (\x+0.5,3+\y+3) {$G$};
\node at (\x+4.5,\y+5.4) {$Y_1$};
\node at (\x+1.5,\y+4.4) {$X_1$};
\node at (\x+4.5,\y+2.6) {$Y_2$};
\node at (\x+5.5,\y+4.4) {$X_2$};
\node at (\x+4.5,\y-0.4) {$Y_3$};
\node at (\x+9.5,\y+4.4) {$X_3$};
\node at (\x+13.5,\y+4.4) {$X_4$};
\node at (\x+5.5,\y-1-2) {$G/X_1^4$};
\node at (\x+0.5,\y-5-1) {$G/Y_1^3$};
\node at (\x+6,\y-5-1) {$Z=V(G/Y_1^3\boxtimes G/X_1^4)$};
\node at (\x+7.5,\y-3.1-2) {$G/Y_1^3\boxtimes G/X_1^4$};
  \def\x{0.5}
  \def\y{6.0+1}
  \Vertex[x=\x+1.5, y=\y+0.0,L={u_2}]{u_2}
  \Vertex[x=\x+1.5, y=\y-3.0,L={u_3}]{u_3}
  \Vertex[x=\x+1.5, y=\y-6.0,L={u_4}]{u_4}
  \Vertex[x=\x+5.5, y=\y+0.0,L={u_5}]{u_5}
  \Vertex[x=\x+5.5, y=\y-3,L={u_6}]{u_6}
  \Vertex[x=\x+5.5, y=\y-6,L={u_7}]{u_7}
  \Vertex[x=\x+9.5, y=\y+0.0,L={u_8}]{u_8}
  \Vertex[x=\x+9.5, y=\y-3.0,L={u_9}]{u_9}
  \Vertex[x=\x+9.5, y=\y-6.0,L={u_{10}}]{u_10}
  \Vertex[x=\x+13.5, y=\y+0,L={u_{11}}]{u_11}
  \Vertex[x=\x+13.5, y=\y-3,L={u_{12}}]{u_12}
  \Vertex[x=\x+13.5, y=\y-6,L={u_{13}}]{u_13}

  \Edge[label = b](u_2)(u_3) 
  \Edge[label = c](u_3)(u_4) 
  \Edge[label = b](u_5)(u_6) 
  \Edge[label = c](u_6)(u_7) 
  \Edge[label = b](u_8)(u_9) 
  \Edge[label = c](u_9)(u_10) 
  \Edge[label = b](u_11)(u_12) 
  \Edge[label = c](u_12)(u_13) 
  \Edge[label = d](u_2)(u_5) 
  \Edge[label = e](u_5)(u_8) 
  \Edge[label = f](u_8)(u_11) 
  \Edge[label = d](u_3)(u_6) 
  \Edge[label = e](u_6)(u_9) 
  \Edge[label = f](u_9)(u_12) 
  \Edge[label = d](u_4)(u_7) 
  \Edge[label = e](u_7)(u_10) 
  \Edge[label = f](u_10)(u_13) 
  
  \Edge(u_2)(u_3) 
  \Edge(u_3)(u_4) 
  \Edge(u_5)(u_6) 
  \Edge(u_6)(u_7) 
  \Edge(u_8)(u_9) 
  \Edge(u_9)(u_10) 
  \Edge(u_11)(u_12) 
  \Edge(u_12)(u_13) 
  \Edge(u_2)(u_5) 
  \Edge(u_5)(u_8) 
  \Edge(u_8)(u_11) 
  \Edge(u_3)(u_6) 
  \Edge(u_6)(u_9) 
  \Edge(u_9)(u_12) 
  \Edge(u_4)(u_7) 
  \Edge(u_7)(u_10) 
  \Edge(u_10)(u_13) 
  
  \def\x{4.0}
  \def\y{-3}
  \Vertex[x=\x+1, y=\y+0.0,L={\tilde{x_1}}]{s_1}
  \Vertex[x=\x+4, y=\y+0.0,L={\tilde{x_2}}]{s_2}
  \Vertex[x=\x+7, y=\y+0.0,L={\tilde{x_3}}]{s_3}
  \Vertex[x=\x+10, y=\y+0.0,L={\tilde{x_4}}]{s_4}
  
  \Edge[label = d](s_1)(s_2) 
  \Edge[label = e](s_2)(s_3) 
  \Edge[label = f](s_3)(s_4) 

  \def\x{+1.5}
  \def\y{-3.0}
  \Vertex[x=\x+0, y=\y-3.0,L={\tilde{y}_1}]{t_1}
  \Vertex[x=\x+0, y=\y-6.0,L={\tilde{y}_2}]{t_2}
  \Vertex[x=\x+0, y=\y-9.0,L={\tilde{y}_3}]{t_3}

  \Edge[label = b](t_1)(t_2) 
  \Edge[label = c](t_2)(t_3) 
  \Edge(t_1)(t_2) 
  \Edge(t_2)(t_3) 

\tikzset{VertexStyle/.append style={
  font=\itshape\large,shape = rounded rectangle,inner sep = 0pt, outer sep = 0pt,minimum size = 20 pt,draw}}

  \def\x{2.0}
  \def\y{-7.0}

  \def\x{2.0}
  \def\y{-6.0}
  \Vertex[x=\x+3.0, y=\y-0.0,L={(\tilde{y}_1,\tilde{x}_1)}]{t_1s_1}
  \Vertex[x=\x+6.0, y=\y-0.0,L={(\tilde{y}_1,\tilde{x}_2)}]{t_1s_2}
  \Vertex[x=\x+9.0, y=\y-0.0,L={(\tilde{y}_1,\tilde{x}_3)}]{t_1s_3}
  \Vertex[x=\x+12.0, y=\y-0.0,L={(\tilde{y}_1,\tilde{x}_4)}]{t_1s_4}
  \def\x{2.0}
  \def\y{-9.0}
  \Vertex[x=\x+3.0, y=\y-0.0,L={(\tilde{y}_2,\tilde{x}_1)}]{t_2s_1}
  \Vertex[x=\x+6.0, y=\y-0.0,L={(\tilde{y}_2,\tilde{x}_2)}]{t_2s_2}
  \Vertex[x=\x+9.0, y=\y-0.0,L={(\tilde{y}_2,\tilde{x}_3)}]{t_2s_3}
  \Vertex[x=\x+12.0, y=\y-0.0,L={(\tilde{y}_2,\tilde{x}_4)}]{t_2s_4}
  \def\x{2.0}
  \def\y{-12.0}
  \Vertex[x=\x+3.0, y=\y-0.0,L={(\tilde{y}_3,\tilde{x}_1)}]{t_3s_1}
  \Vertex[x=\x+6.0, y=\y-0.0,L={(\tilde{y}_3,\tilde{x}_2)}]{t_3s_2}
  \Vertex[x=\x+9.0, y=\y-0.0,L={(\tilde{y}_3,\tilde{x}_3)}]{t_3s_3}
  \Vertex[x=\x+12.0, y=\y-0.0,L={(\tilde{y}_3,\tilde{x}_4)}]{t_3s_4}
    \def\x{2.0}
  \def\y{-19.0}


  \Edge[label = d](t_1s_1)(t_1s_2) 
  \Edge[label = e](t_1s_2)(t_1s_3) 
  \Edge[label = f](t_1s_3)(t_1s_4) 
  \Edge(t_1s_1)(t_1s_2) 
  \Edge(t_1s_2)(t_1s_3) 
  \Edge(t_1s_3)(t_1s_4) 

  \Edge[label = d](t_2s_1)(t_2s_2) 
  \Edge[label = e](t_2s_2)(t_2s_3) 
  \Edge[label = f](t_2s_3)(t_2s_4) 
  \Edge(t_2s_1)(t_2s_2) 
  \Edge(t_2s_2)(t_2s_3) 
  \Edge(t_2s_3)(t_2s_4) 

  \Edge[label = d](t_3s_1)(t_3s_2) 
  \Edge[label = e](t_3s_2)(t_3s_3) 
  \Edge[label = f](t_3s_3)(t_3s_4) 
  \Edge(t_3s_1)(t_3s_2) 
  \Edge(t_3s_2)(t_3s_3) 
  \Edge(t_3s_3)(t_3s_4) 


  \Edge[label = b](t_1s_1)(t_2s_1) 
  \Edge[label = b](t_1s_2)(t_2s_2) 
  \Edge[label = b](t_1s_3)(t_2s_3)
  \Edge[label = b](t_1s_4)(t_2s_4)  
  \Edge(t_1s_1)(t_2s_1) 
  \Edge(t_1s_2)(t_2s_2) 
  \Edge(t_1s_3)(t_2s_3)
  \Edge(t_1s_4)(t_2s_4)  
 
  \Edge[label = c](t_2s_1)(t_3s_1) 
  \Edge[label = c](t_2s_2)(t_3s_2) 
  \Edge[label = c](t_2s_3)(t_3s_3)
  \Edge[label = c](t_2s_4)(t_3s_4)  
  \Edge(t_2s_1)(t_3s_1) 
  \Edge(t_2s_2)(t_3s_2) 
  \Edge(t_2s_3)(t_3s_3)
  \Edge(t_2s_4)(t_3s_4)  


  \def\x{1.7}
  \def\y{5.9+1}
\draw[circle, -,dashed, very thick,rounded corners=18pt] (\x-0.5,\y+0.0)--(\x-0.5,\y+0.7) --(\x+1.0,\y+0.7) -- (\x+1.0,\y-6.6) -- (\x-0.5,\y-6.6) --  (\x-0.5,\y+0.0);
  \def\x{5.7}
  \def\y{5.9+1}
\draw[circle, -,dashed, very thick,rounded corners=18pt] (\x-0.5,\y+0.0)--(\x-0.5,\y+0.7) --(\x+1.0,\y+0.7) -- (\x+1.0,\y-6.6) -- (\x-0.5,\y-6.6) --  (\x-0.5,\y+0.0);
  \def\x{9.7}
  \def\y{5.9+1}
\draw[circle, -,dashed, very thick,rounded corners=18pt] (\x-0.5,\y+0.0)--(\x-0.5,\y+0.7) --(\x+1.0,\y+0.7) -- (\x+1.0,\y-6.6) -- (\x-0.5,\y-6.6) --  (\x-0.5,\y+0.0);
  \def\x{13.7}
  \def\y{5.9+1}
\draw[circle, -,dashed, very thick,rounded corners=18pt] (\x-0.5,\y+0.0)--(\x-0.5,\y+0.7) --(\x+1.0,\y+0.7) -- (\x+1.0,\y-6.6) -- (\x-0.5,\y-6.6) --  (\x-0.5,\y+0.0);

  \def\x{1.7}
  \def\y{5.1+1}
\draw[circle, -,dashed, very thick,rounded corners=18pt] (\x+2.8,\y+0.0)--(\x-0.5,\y+0.0)--(\x-0.5,\y+1.5)--(\x+13.0,\y+1.5) --(\x+13.0,\y) -- (\x+2.8,\y+0.0);

  \def\x{1.7}
  \def\y{2.3+1}
\draw[circle, -,dashed, very thick,rounded corners=18pt] (\x+2.8,\y+0.0)--(\x-0.5,\y+0.0)--(\x-0.5,\y+1.5)--(\x+13.0,\y+1.5) --(\x+13.0,\y) -- (\x+2.8,\y+0.0);

  \def\x{1.7}
  \def\y{-0.7+1}
\draw[circle, -,dashed, very thick,rounded corners=18pt] (\x+2.8,\y+0.0)--(\x-0.5,\y+0.0)--(\x-0.5,\y+1.5)--(\x+13.0,\y+1.5) --(\x+13.0,\y) -- (\x+2.8,\y+0.0);

    \def\x{0.5}
  \def\y{6.2+1}

  \def\x{1.5}
  \def\y{-6.0}

  \def\x{3.5}
  \def\y{-6.4}
\draw[circle, -,dashed, very thick,rounded corners=8pt] (\x+0.2,\y+2)--(\x+11.4,\y+2) --(\x+11.9,\y+1.5) -- (\x+11.9,\y-7)-- (\x+11.4,\y-7.5) -- (\x-0.3,\y-7.5) -- (\x-0.8,\y-7) -- (\x-0.8,\y+1.5) -- (\x-0.3,\y+2)--(\x+0.1,\y+2);
\end{tikzpicture}
}
\end{center}
\caption{Decomposition of $G\cong G/_{i=1}^3Y_i\boxbackslash G/_{i=1}^4X_i$. The set $Z$ from the proof of Theorem~\ref{theorem_6} and the graph isomorphic to $G$ induced by $Z$ in $G/_{i=1}^3Y_i\boxtimes G/_{i=1}^4X_i$ are indicated within the dotted region.}
  \label{ThirdDecomposition}
\end{figure}
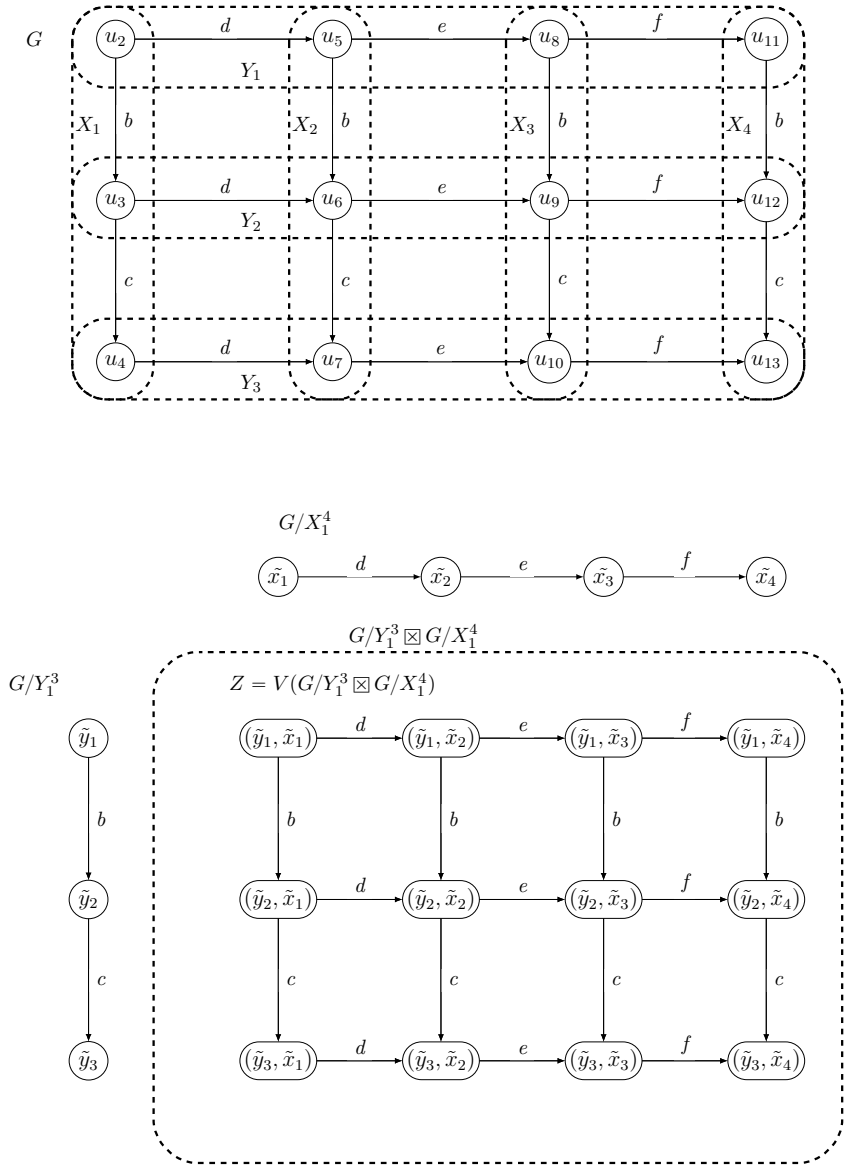


\begin{theorem}\label{theorem_6}
Let $G$ be a Cartesian matrix graph where $V(G)$ consists of nonempty pairwise disjoint subsets $R_i=\{v_{i,j}\mid j\in J_C\subseteq J\},i\in I_R\subseteq I$, and nonempty pairwise disjoint subsets $C_j=\{v_{i,j}\mid i\in I_R\subseteq I\},j\in J_C\subseteq J,|I_R|=m_1,|J_C|=n_1,$, with $V(G)=\bigcup\limits_{i\in I_R}R_i=\bigcup\limits_{j\in J_C}C_j$, for which $G[R_x]\cong G[R_y],x,y\in I_R$, $G[C_x]\cong G[C_y],x,y\in J_C$, the arcs of $A_R=\bigcup\limits_{x\in I_R}A[R_x]$ and the arcs of $A_C=\bigcup\limits_{y\in I_C}A[C_y]$ have no labels in common and there are no other arcs in $A(G)$ than the arcs of $A_R$ and the arcs of $A_C$.
Then $G/_{i\in I_R}R_i\boxbackslash G/_{j\in J_C}C_j\cong G$.
\end{theorem}
\begin{proof}
It clearly suffices to define a mapping $\phi: V(G)\rightarrow V(G/_{i\in I_R}R_i\boxbackslash G/_{j\in J_C}C_j)$ and to prove that $\phi$ is an isomorphism from $G$ to $G/_{i\in I_R}R_i\boxbackslash G/_{j\in J_C}C_j$.

Let $\tilde{x}'_i$ be the new vertex replacing the set $R_i$ and let $\tilde{x}''_j$  be the new vertex replacing the set $C_j$, when defining $G/_{i\in I_R}R_i$ and $G/_{j\in J_C}C_j$, respectively.
Consider the mapping $\phi: V(G)\rightarrow V(G/_{i\in I_R}R_i\boxbackslash G/_{j\in J_C}C_j)$ defined by  $\phi(v_{i,j})=(\tilde{x}'_i,\tilde{x}''_j)$ for all $v_{i,j}\in V(G)$.
   
\noindent Then $\phi$ is obviously a bijection if  $V(G/_{i\in I_R}R_i\boxbackslash G/_{j\in J_C}C_j)=Z$, where $Z$ is defined as $Z=\{(\tilde{x}'_i,\tilde{x}''_j)\mid \phi(v_{i,j})=(\tilde{x}_i,\tilde{x}_j), v_{i,j}\in V(G)\}$.
Furthermore, the set of vertices of $Z$ is identical to the set of vertices of $G/_{i\in I_R}R_i\boxtimes G/_{j\in J_C}C_j$. 

We start with proving that the contraction of $R_i$ to $\tilde{x}'_i$ and the contraction of $R_j$  to $\tilde{x}'_j$ for $i\neq j,$ implies $\tilde{x}'_i\neq \tilde{x}'_j$ and that the contraction of $C_j$ to $\tilde{x}''_j$ and the contraction of $C_k$  to $\tilde{x}''_k$ for $j\neq k,$ implies $\tilde{x}''_j\neq \tilde{x}''_k$.
Because $R_i$ is the row with all vertices $v_{i,k}$ of $V(G)$ (by hypothesis) the vertices of $R_i$ are replaced by $\tilde{x}'_i$ and $R_j$ is the row with all vertices $v_{j,k}$ of $V(G)$ (by hypothesis) the vertices of $R_j$ are replaced by $\tilde{x}'_j$, and $R_i\cap R_j=\emptyset$ for $i\neq j$, we have that the contraction of $R_i$ to $\tilde{x}'_i$ and the contraction of $R_j$  to $\tilde{x}'_j$ for $i\neq j$ implies $\tilde{x}'_i\neq \tilde{x}'_j$. 
Likewise, Because $C_j$ is the column with all vertices $v_{i,j}$ of $V(G)$ (by hypothesis) the vertices of $C_j$ are replaced by $\tilde{x}''_j$ and $C_k$ is the column with all vertices $v_{i,k}$ of $V(G)$ (by hypothesis) the vertices of $C_k$ are replaced by $\tilde{x}''_k$, and $C_j\cap C_k=\emptyset$ for $j\neq k$, we have that the contraction of $C_k$  to $\tilde{x}''_k$ for $j\neq k,$ implies $\tilde{x}''_j\neq \tilde{x}''_k$. 

Because $Z$ consists of vertices $(\tilde{x}'_i,\tilde{x}''_j)$ only and $\phi$ maps $v_{i,j}$ onto $(\tilde{x}'_i,\tilde{x}''_j)$, and if $v_{i_1,j_1}\neq v_{i_2,j_2}$ then  $(\tilde{x}'_{i_1},\tilde{x}''_{i_2})\neq (\tilde{x}'_{j_1}\tilde{x}''_{j_2})$, $v_{i_1,j_1}, v_{i_2,j_2}\in V(G)$ we have that $\phi$ is a bijection from $V(G)$ to $Z$. 
It remains to show that this bijection preserves the arcs and their labels. 
By hypothesis, the arcs of the rows $R_i$ of $G$ are asynchronous with respect to the arcs of the columns $C_j$ of $G$ and by hypothesis we have only arcs $a\in A(G)$ with $\mu(a)=(u_{i,j},u_{i,k})$ for $u_{i,j}\in R_i$, $u_{i,k}\in R_i$ and arcs $a\in A(G)$ with $\mu(a)=(u_{i,k},u_{j,k})$ for $u_{i,k}\in C_k$, $u_{j,k}\in C_k$.
Hence, together with the definition of the Cartesian product, for each arc $a\in A(G)$ with $\mu(a)=(u_{i,j},u_{i,k})$ for $u_{i,j}\in R_i$, $u_{i,k}\in R_i$, there exists an arc $b$ in $G/_{i\in I_R}R_i\boxtimes G/_{j\in J_C}C_j$ with $\mu(b)=((\tilde{x}'_i,\tilde{x}''_j),(\tilde{x}'_i,\tilde{x}'_k))=(\phi(u_{i,j}),\phi(u_{i,k}))$ and $\lambda(b)=\lambda(a)$. 
Likewise, for each arc $a\in A(G)$ with $\mu(a)=(u_{i,k},u_{j,k})$ for $u_{i,k}\in C_k$, $u_{j,k}\in C_k$, there exists an arc $b$ in $G/_{i\in I_R}R_i\boxtimes G/_{j\in J_C}C_j$ with $\mu(b)=((\tilde{x}'_i,\tilde{x}''_k),(\tilde{x}'_j,\tilde{x}''_k))=(\phi(u_{i,k}),\phi(u_{j,k}))$ and $\lambda(b)=\lambda(a)$.

Because $G$ is acyclic, the above arcs are the only arcs in $G/_{i\in I_R}R_i\boxtimes G/_{j\in J_C}C_j$ induced by the vertices of $Z$.
Furthermore, there are no other vertices in $G/_{i\in I_R}R_i\boxtimes G/_{j\in J_C}C_j$ than the vertices of $Z$, because all vertices of $Z$ are of the type $(\tilde{x}'_i,\tilde{x}''_j)$ (for the head and the tail of asynchronous arcs).
This completes the proof of Theorem~\ref{theorem_6}. 
\end{proof}
Note that the decomposition by Theorem~\ref{theorem_6} iteratively decomposes any graph $G$ that is the product of graphs $G_1,\ldots,G_n$, $G\cong \overundersyncprod{i=1}{{n}}G_i$, that do not share a label.
We call a matrix graph that is decomposable by Theorem~\ref{theorem_6} a \emph{VRSP-decomposable Cartesian matrix graph} and we call a subgraph $G'$ of a matrix graph $G$ a \emph{maximal VRSP-decomposable Cartesian matrix subgraph} if $G'$ is a VRSP-decomposable Cartesian matrix graph and there is no subgraph $G''$ of $G$ where $G''$ is a VRSP-decomposable Cartesian matrix graph and $G'$ is a proper subgraph of $G''$.

We continue with a decomposition theorem where we use implicitly both Theorem~\ref{theorem_5} and Theorem~\ref{theorem_6}.
The graphs containing maximal VRSP-decomposable 
 Cartesian matrix subgraphs and VRSP-decomposable 
 semicomplete bipartite matrix subgraphs cannot be decomposed by either Theorem~\ref{theorem_5} or Theorem~\ref{theorem_6}.
In Figure~\ref{FourthDecomposition}, we give an example where the vertices are numbered according to the matrix scheme for maximal VRSP-decomposable 
 Cartesian matrix subgraphs and VRSP-decomposable 
 semicomplete bipartite matrix subgraphs.
 This scheme leads to five rows and six columns for which the contraction of the rows produces the graph $G/_{i=1}^5R_i$ and the contraction of the columns produces the graph $G/_{i=1}^6C_i$. 
 The VRSP of these two graphs gives the graph $G/_{i=1}^5R_i\boxbackslash G/_{i=1}^6C_i$ which is isomorphic to $G$.
In Theorem~\ref{theorem_7}, we state and proof the scheme described in Figure~\ref{FourthDecomposition}.
 
 \begin{theorem}\label{theorem_7}
Let $G$ be a matrix graph consisting solely of a set of
 maximal VRSP-decomposable 
 Cartesian matrix subgraphs $G_{M}$ of $G$ and a set of 
 VRSP-decomposable 
 semicomplete bipartite matrix subgraphs $G_B$ of $G$ where each semicomplete bipartite subgraph is arc-induced by a set of all arcs of $G$ with identical labels.
Let any subgraph $G_{M_{1}}$ of $G_M$ and any subgraph $G_{M_{2}}$ of $G_M$ with $ V(G_{M_1})\cap V(G_{M_2})=\emptyset$ and the subgraphs of $G_B$ have no labels in common.
Let there be no arc $a$ of $G_B$ with $\mu(a)=v_{i,j}v_{i,k}$ and $v_{i,j},v_{i,k}$ in any $V(G_{M_x})$ of $G_M$ and let there be no arc $a$ of $G_B$ with $\mu(a)=v_{i,j}v_{k,j}$ and $v_{i,j},v_{k,j}$ in any $V(G_{M_y})$ of $G_M$.
If each row $R_x$ of $G$ that contains the vertex $v_{i,j}$ has the index $i$ and
if each column $C_y$ of $G$ that contains the vertex $v_{i,j}$ has the index $j$ then $G\cong G/_{i=1}^{m}R_i\boxbackslash G/_{j=1}^{n}C_j$.
\end{theorem}

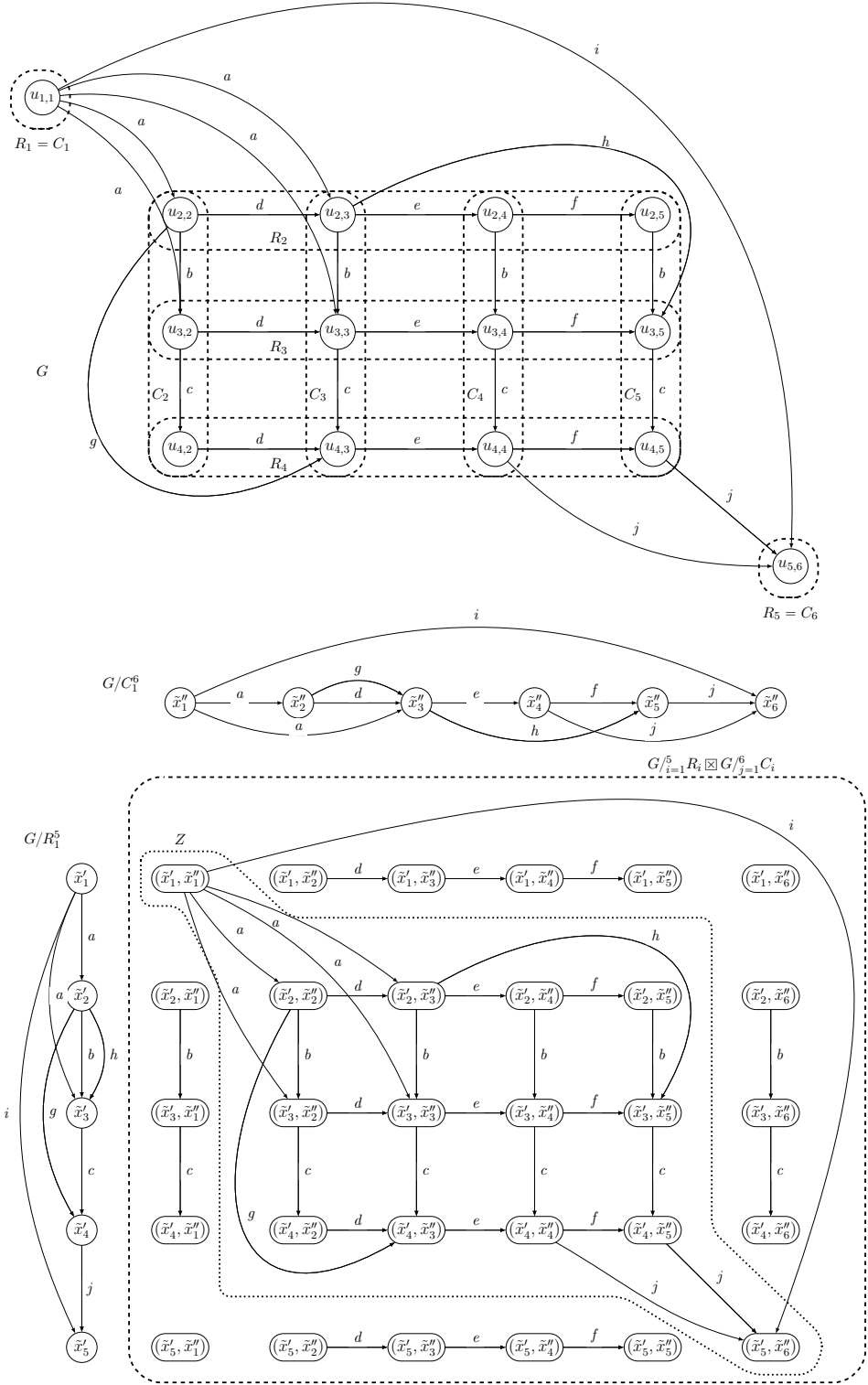
\begin{figure}[H]
\begin{center}
\resizebox{0.95\textwidth}{!}{
\begin{tikzpicture}[->,>=latex,shorten >=0pt,auto,node distance=2.5cm,
  main node/.style={circle,fill=blue!10,draw, font=\sffamily\Large\bfseries}]
  \tikzset{VertexStyle/.append style={
  font=\itshape\large, shape = circle,inner sep = 2pt, outer sep = 0pt,minimum size = 20 pt,draw}}
  \tikzset{EdgeStyle/.append style={thin}}
  \tikzset{LabelStyle/.append style={font = \itshape}}
  \SetVertexMath
  \def\x{0.0}
  \def\y{1.0}
\node at (\x-1.5,3+\y+2) {$G$};
\node at (\x-1.5,\y+10.8) {$R_1=C_1$};
\node at (\x-1.5+19,\y-1.2) {$R_5=C_6$};
\node at (\x+4.5,\y+8.4) {$R_2$};
\node at (\x+1.5,\y+4.4) {$C_2$};
\node at (\x+4.5,\y+5.6) {$R_3$};
\node at (\x+5.5,\y+4.4) {$C_3$};
\node at (\x+4.5,\y+3-0.4) {$R_4$};
\node at (\x+9.5,\y+4.4) {$C_4$};
\node at (\x+13.5,\y+4.4) {$C_5$};
\node at (\x+0.5,\y-1-2) {$G/C_1^6$};
\node at (\x-1.5,\y-5-2) {$G/R_1^5$};
\node at (\x+2,\y-5-2) {$Z$};
\node at (\x+15.5,\y-3.1-2) {$G/_{i=1}^5R_i\boxtimes G/_{j=1}^6C_i$};
  \def\x{0.5}
  \def\y{9.0+1}
  \Vertex[x=\x+-2, y=\y+3.0,L={u_{1,1}}]{u_1}
  \Vertex[x=\x+1.5, y=\y+0.0,L={u_{2,2}}]{u_2}
  \Vertex[x=\x+1.5, y=\y-3.0,L={u_{3,2}}]{u_3}
  \Vertex[x=\x+1.5, y=\y-6.0,L={u_{4,2}}]{u_4}
  \Vertex[x=\x+5.5, y=\y+0.0,L={u_{2,3}}]{u_5}
  \Vertex[x=\x+5.5, y=\y-3,L={u_{3,3}}]{u_6}
  \Vertex[x=\x+5.5, y=\y-6,L={u_{4,3}}]{u_7}
  \Vertex[x=\x+9.5, y=\y+0.0,L={u_{2,4}}]{u_8}
  \Vertex[x=\x+9.5, y=\y-3.0,L={u_{3,4}}]{u_9}
  \Vertex[x=\x+9.5, y=\y-6.0,L={u_{4,4}}]{u_10}
  \Vertex[x=\x+13.5, y=\y+0,L={u_{2,5}}]{u_11}
  \Vertex[x=\x+13.5, y=\y-3,L={u_{3,5}}]{u_12}
  \Vertex[x=\x+13.5, y=\y-6,L={u_{4,5}}]{u_13}
  \Vertex[x=\x+17, y=\y-9.0,L={u_{5,6}}]{u_14}

  \Edge[label = b](u_2)(u_3) 
  \Edge[label = c](u_3)(u_4) 
  \Edge[label = b](u_5)(u_6) 
  \Edge[label = c](u_6)(u_7) 
  \Edge[label = b](u_8)(u_9) 
  \Edge[label = c](u_9)(u_10) 
  \Edge[label = b](u_11)(u_12) 
  \Edge[label = c](u_12)(u_13) 
  \Edge[label = d](u_2)(u_5) 
  \Edge[label = e](u_5)(u_8) 
  \Edge[label = f](u_8)(u_11) 
  \Edge[label = d](u_3)(u_6) 
  \Edge[label = e](u_6)(u_9) 
  \Edge[label = f](u_9)(u_12) 
  \Edge[label = d](u_4)(u_7) 
  \Edge[label = e](u_7)(u_10) 
  \Edge[label = f](u_10)(u_13) 
  \Edge[label = j](u_13)(u_14) 
  
  \Edge(u_2)(u_3) 
  \Edge(u_3)(u_4) 
  \Edge(u_5)(u_6) 
  \Edge(u_6)(u_7) 
  \Edge(u_8)(u_9) 
  \Edge(u_9)(u_10) 
  \Edge(u_11)(u_12) 
  \Edge(u_12)(u_13) 
  \Edge(u_2)(u_5) 
  \Edge(u_5)(u_8) 
  \Edge(u_8)(u_11) 
  \Edge(u_3)(u_6) 
  \Edge(u_6)(u_9) 
  \Edge(u_9)(u_12) 
  \Edge(u_4)(u_7) 
  \Edge(u_7)(u_10) 
  \Edge(u_10)(u_13) 
  \Edge(u_13)(u_14)   
  
 \Edge[label = j, labelstyle={xshift=0pt, yshift=2pt}, style={in = -180, out = -45,min distance=1cm}](u_10)(u_14)
\Edge[label = i, labelstyle={xshift=0pt, yshift=2pt}, style={bend right=-60,min distance=12cm}](u_1)(u_14)
\Edge[label = a, labelstyle={xshift=0pt, yshift=2pt}, style={bend right=-30,min distance=1cm}](u_1)(u_2)
\Edge[label = a, labelstyle={xshift=-30pt, yshift=-5pt}, style={bend right=-30,min distance=1cm}](u_1)(u_3)
\Edge[label = a, labelstyle={xshift=0pt, yshift=2pt}, style={bend right=-45,min distance=1cm}](u_1)(u_5)
\Edge[label = a, labelstyle={xshift=0pt, yshift=2pt}, style={bend right=-45,min distance=2cm}](u_1)(u_6)
\Edge[label = g, labelstyle={xshift=-16pt, yshift=-4pt}, style={in = 210, out=225,min distance=6cm}](u_2)(u_7)
  \def\x{4.0}
  \def\y{-2.5}
  \Vertex[x=\x+-2, y=\y+0.0,L={\tilde{x}''_1}]{u_1}
  \Vertex[x=\x+1, y=\y+0.0,L={\tilde{x}''_2}]{s_1}
  \Vertex[x=\x+4, y=\y+0.0,L={\tilde{x}''_3}]{s_2}
  \Vertex[x=\x+7, y=\y+0.0,L={\tilde{x}''_4}]{s_3}
  \Vertex[x=\x+10, y=\y+0.0,L={\tilde{x}''_5}]{s_4}
  \Vertex[x=\x+13, y=\y+0.0,L={\tilde{x}''_6}]{u_14}
  
  \Edge[label = a](u_1)(s_1) 
  \Edge[label = d](s_1)(s_2) 
  \Edge[label = e](s_2)(s_3) 
  \Edge[label = f](s_3)(s_4) 
  \Edge[label = j](s_4)(u_14) 
  \Edge[label = j, labelstyle={xshift=0pt, yshift=0pt}, style={in = -150, out = -30,min distance=1cm}](s_3)(u_14)
\Edge[label = i, labelstyle={xshift=0pt, yshift=2pt}, style={bend right=-25,min distance=1cm}](u_1)(u_14)
\Edge[label = a, labelstyle={xshift=0pt, yshift=0pt}, style={bend left=-25,min distance=1cm}](u_1)(s_2)
\Edge[label = g, labelstyle={xshift=0pt, yshift=-0pt}, style={in = 150, out=30,min distance=1cm}](s_1)(s_2)
\Edge[label = h, labelstyle={xshift=0pt, yshift=-0pt}, style={in = -150, out=-30,min distance=1cm}](s_2)(s_4)
\Edge[style={in = 150, out=30,min distance=1cm}](s_1)(s_2)
\Edge[style={in = -150, out=-30,min distance=1cm}](s_2)(s_4)

  \def\x{-0.5}
  \def\y{-7.0}
  \Vertex[x=\x+0, y=\y+0.0,L={\tilde{x}'_1}]{u_1}
  \Vertex[x=\x+0, y=\y-3.0,L={\tilde{x}'_2}]{t_1}
  \Vertex[x=\x+0, y=\y-6.0,L={\tilde{x}'_3}]{t_2}
  \Vertex[x=\x+0, y=\y-9.0,L={\tilde{x}'_4}]{t_3}
  \Vertex[x=\x+0, y=\y-12.0,L={\tilde{x}'_5}]{u_14}

  \Edge[label = a](u_1)(t_1) 
  \Edge[label = b](t_1)(t_2) 
  \Edge[label = c](t_2)(t_3) 
  \Edge[label = j](t_3)(u_14) 
\Edge[label = i, labelstyle={xshift=-16pt, yshift=0pt}, style={bend right=25,min distance=1cm}](u_1)(u_14)
\Edge[label = a, labelstyle={xshift=0pt, yshift=0pt}, style={bend right=25,min distance=1cm}](u_1)(t_2)
\Edge[label = h, labelstyle={xshift=-16pt, yshift=-0pt}, style={in = 60, out=30,min distance=5.5cm}](u_5)(u_12)
  \Edge(u_1)(t_1) 
  \Edge(t_1)(t_2) 
  \Edge(t_2)(t_3) 
  \Edge(t_3)(u_14) 
  \Edge[style={in = 210, out=225,min distance=6cm}](u_2)(u_7)
  \Edge[style={in = 60, out=30,min distance=5.5cm}](u_5)(u_12)
\Edge[label = g, labelstyle={xshift=0pt, yshift=-0pt}, style={in = 120, out=-120,min distance=1cm}](t_1)(t_3)
\Edge[label = h, labelstyle={xshift=0pt, yshift=-0pt}, style={in = 60, out=-60,min distance=1cm}](t_1)(t_2)
\Edge[style={in = 120, out=-120,min distance=1cm}](t_1)(t_3)
\Edge[style={in = 60, out=-60,min distance=1cm}](t_1)(t_2)

\tikzset{VertexStyle/.append style={
  font=\itshape\large,shape = rounded rectangle,inner sep = 0pt, outer sep = 0pt,minimum size = 20 pt,draw}}

  \def\x{2.0}
  \def\y{-7.0}
  \Vertex[x=\x+0.0, y=\y-0.0,L={(\tilde{x}'_1,\tilde{x}''_1)}]{u_1u_1}
  \Vertex[x=\x+3.0, y=\y-0.0,L={(\tilde{x}'_1,\tilde{x}''_2)}]{u_1s_1}
  \Vertex[x=\x+6.0, y=\y-0.0,L={(\tilde{x}'_1,\tilde{x}''_3)}]{u_1s_2}
  \Vertex[x=\x+9.0, y=\y-0.0,L={(\tilde{x}'_1,\tilde{x}''_4)}]{u_1s_3}
  \Vertex[x=\x+12.0, y=\y-0.0,L={(\tilde{x}'_1,\tilde{x}''_5)}]{u_1s_4}
  \Vertex[x=\x+15.0, y=\y-0.0,L={(\tilde{x}'_1,\tilde{x}''_6)}]{u_1u_14}

  \def\x{2.0}
  \def\y{-10.0}
  \Vertex[x=\x+0.0, y=\y-0.0,L={(\tilde{x}'_2,\tilde{x}''_1)}]{t_1u_1}
  \Vertex[x=\x+3.0, y=\y-0.0,L={(\tilde{x}'_2,\tilde{x}''_2)}]{t_1s_1}
  \Vertex[x=\x+6.0, y=\y-0.0,L={(\tilde{x}'_2,\tilde{x}''_3)}]{t_1s_2}
  \Vertex[x=\x+9.0, y=\y-0.0,L={(\tilde{x}'_2,\tilde{x}''_4)}]{t_1s_3}
  \Vertex[x=\x+12.0, y=\y-0.0,L={(\tilde{x}'_2,\tilde{x}''_5)}]{t_1s_4}
  \Vertex[x=\x+15.0, y=\y-0.0,L={(\tilde{x}'_2,\tilde{x}''_6)}]{t_1u_14}
  \def\x{2.0}
  \def\y{-13.0}
  \Vertex[x=\x+0.0, y=\y-0.0,L={(\tilde{x}'_3,\tilde{x}''_1)}]{t_2u_1}
  \Vertex[x=\x+3.0, y=\y-0.0,L={(\tilde{x}'_3,\tilde{x}''_2)}]{t_2s_1}
  \Vertex[x=\x+6.0, y=\y-0.0,L={(\tilde{x}'_3,\tilde{x}''_3)}]{t_2s_2}
  \Vertex[x=\x+9.0, y=\y-0.0,L={(\tilde{x}'_3,\tilde{x}''_4)}]{t_2s_3}
  \Vertex[x=\x+12.0, y=\y-0.0,L={(\tilde{x}'_3,\tilde{x}''_5)}]{t_2s_4}
  \Vertex[x=\x+15.0, y=\y-0.0,L={(\tilde{x}'_3,\tilde{x}''_6)}]{t_2u_14}
  \def\x{2.0}
  \def\y{-16.0}
  \Vertex[x=\x+0.0, y=\y-0.0,L={(\tilde{x}'_4,\tilde{x}''_1)}]{t_3u_1}
  \Vertex[x=\x+3.0, y=\y-0.0,L={(\tilde{x}'_4,\tilde{x}''_2)}]{t_3s_1}
  \Vertex[x=\x+6.0, y=\y-0.0,L={(\tilde{x}'_4,\tilde{x}''_3)}]{t_3s_2}
  \Vertex[x=\x+9.0, y=\y-0.0,L={(\tilde{x}'_4,\tilde{x}''_4)}]{t_3s_3}
  \Vertex[x=\x+12.0, y=\y-0.0,L={(\tilde{x}'_4,\tilde{x}''_5)}]{t_3s_4}
  \Vertex[x=\x+15.0, y=\y-0.0,L={(\tilde{x}'_4,\tilde{x}''_6)}]{t_3u_14}
    \def\x{2.0}
  \def\y{-19.0}
  \Vertex[x=\x+0.0, y=\y-0.0,L={(\tilde{x}'_5,\tilde{x}''_1)}]{u_14u_1}
  \Vertex[x=\x+3.0, y=\y-0.0,L={(\tilde{x}'_5,\tilde{x}''_2)}]{u_14s_1}
  \Vertex[x=\x+6.0, y=\y-0.0,L={(\tilde{x}'_5,\tilde{x}''_3)}]{u_14s_2}
  \Vertex[x=\x+9.0, y=\y-0.0,L={(\tilde{x}'_5,\tilde{x}''_4)}]{u_14s_3}
  \Vertex[x=\x+12.0, y=\y-0.0,L={(\tilde{x}'_5,\tilde{x}''_5)}]{u_14s_4}
  \Vertex[x=\x+15.0, y=\y-0.0,L={(\tilde{x}'_5,\tilde{x}''_6)}]{u_14u_14}

\Edge[label = a, labelstyle={xshift=0pt, yshift=2pt}, style={bend right=10,min distance=1cm}](u_1u_1)(t_1s_1)
\Edge[label = a, labelstyle={xshift=-30pt, yshift=-5pt}, style={bend right=-10,min distance=1cm}](u_1u_1)(t_1s_2)
\Edge[label = a, labelstyle={xshift=0pt, yshift=2pt}, style={bend right=10,min distance=1cm}](u_1u_1)(t_2s_1)
\Edge[label = a, labelstyle={xshift=0pt, yshift=2pt}, style={bend right=-20,min distance=2cm}](u_1u_1)(t_2s_2)

  \Edge[label = j](t_3s_4)(u_14u_14) 
  \Edge[label = d](u_1s_1)(u_1s_2) 
  \Edge[label = e](u_1s_2)(u_1s_3) 
  \Edge[label = f](u_1s_3)(u_1s_4) 
  \Edge(t_3s_4)(u_14u_14) 
  \Edge(u_1s_1)(u_1s_2) 
  \Edge(u_1s_2)(u_1s_3) 
  \Edge(u_1s_3)(u_1s_4) 

  \Edge[label = d](t_1s_1)(t_1s_2) 
  \Edge[label = e](t_1s_2)(t_1s_3) 
  \Edge[label = f](t_1s_3)(t_1s_4) 
  \Edge(t_1s_1)(t_1s_2) 
  \Edge(t_1s_2)(t_1s_3) 
  \Edge(t_1s_3)(t_1s_4) 

  \Edge[label = d](t_2s_1)(t_2s_2) 
  \Edge[label = e](t_2s_2)(t_2s_3) 
  \Edge[label = f](t_2s_3)(t_2s_4) 
  \Edge(t_2s_1)(t_2s_2) 
  \Edge(t_2s_2)(t_2s_3) 
  \Edge(t_2s_3)(t_2s_4) 

  \Edge[label = d](t_3s_1)(t_3s_2) 
  \Edge[label = e](t_3s_2)(t_3s_3) 
  \Edge[label = f](t_3s_3)(t_3s_4) 
  \Edge(t_3s_1)(t_3s_2) 
  \Edge(t_3s_2)(t_3s_3) 
  \Edge(t_3s_3)(t_3s_4) 

  \Edge[label = d](u_14s_1)(u_14s_2) 
  \Edge[label = e](u_14s_2)(u_14s_3) 
  \Edge[label = f](u_14s_3)(u_14s_4) 
  \Edge(u_14s_1)(u_14s_2) 
  \Edge(u_14s_2)(u_14s_3) 
  \Edge(u_14s_3)(u_14s_4)

  \Edge[label = b](t_1u_1)(t_2u_1) 
  \Edge[label = b](t_1s_1)(t_2s_1) 
  \Edge[label = b](t_1s_2)(t_2s_2) 
  \Edge[label = b](t_1s_3)(t_2s_3)
  \Edge[label = b](t_1s_4)(t_2s_4)  
  \Edge[label = b](t_1u_14)(t_2u_14) 
  \Edge(t_1u_1)(t_2u_1) 
  \Edge(t_1s_1)(t_2s_1) 
  \Edge(t_1s_2)(t_2s_2) 
  \Edge(t_1s_3)(t_2s_3)
  \Edge(t_1s_4)(t_2s_4)  
  \Edge(t_1u_14)(t_2u_14) 
 
  \Edge[label = c](t_2u_1)(t_3u_1) 
  \Edge[label = c](t_2s_1)(t_3s_1) 
  \Edge[label = c](t_2s_2)(t_3s_2) 
  \Edge[label = c](t_2s_3)(t_3s_3)
  \Edge[label = c](t_2s_4)(t_3s_4)  
  \Edge[label = c](t_2u_14)(t_3u_14) 
  \Edge(t_2u_1)(t_3u_1) 
  \Edge(t_2s_1)(t_3s_1) 
  \Edge(t_2s_2)(t_3s_2) 
  \Edge(t_2s_3)(t_3s_3)
  \Edge(t_2s_4)(t_3s_4)  
  \Edge(t_2u_14)(t_3u_14)

  \Edge[label = j, labelstyle={xshift=0pt, yshift=0pt}, style={in = 165, out = -30,min distance=1cm}](t_3s_3)(u_14u_14)
 \Edge[label = i, labelstyle={xshift=0pt, yshift=2pt}, style={in = 70, out = 15,min distance=15.1cm}](u_1u_1)(u_14u_14)
\Edge[label = g, labelstyle={xshift=0pt, yshift=-0pt}, style={in = 210, out=240,min distance=5cm}](t_1s_1)(t_3s_2)
\Edge[label = h, labelstyle={xshift=0pt, yshift=-0pt}, style={in = 60, out=30,min distance=4.8cm}](t_1s_2)(t_2s_4)
\Edge[style={in = 210, out=240,min distance=5cm}](t_1s_1)(t_3s_2)
\Edge[style={in = 60, out=30,min distance=4.8cm}](t_1s_2)(t_2s_4)

  \def\x{1.7}
  \def\y{8.9+1}
\draw[circle, -,dashed, very thick,rounded corners=18pt] (\x-0.5,\y+0.0)--(\x-0.5,\y+0.7) --(\x+1.0,\y+0.7) -- (\x+1.0,\y-6.6) -- (\x-0.5,\y-6.6) --  (\x-0.5,\y+0.0);
  \def\x{5.7}
  \def\y{8.9+1}
\draw[circle, -,dashed, very thick,rounded corners=18pt] (\x-0.5,\y+0.0)--(\x-0.5,\y+0.7) --(\x+1.0,\y+0.7) -- (\x+1.0,\y-6.6) -- (\x-0.5,\y-6.6) --  (\x-0.5,\y+0.0);
  \def\x{9.7}
  \def\y{8.9+1}
\draw[circle, -,dashed, very thick,rounded corners=18pt] (\x-0.5,\y+0.0)--(\x-0.5,\y+0.7) --(\x+1.0,\y+0.7) -- (\x+1.0,\y-6.6) -- (\x-0.5,\y-6.6) --  (\x-0.5,\y+0.0);
  \def\x{13.7}
  \def\y{8.9+1}
\draw[circle, -,dashed, very thick,rounded corners=18pt] (\x-0.5,\y+0.0)--(\x-0.5,\y+0.7) --(\x+1.0,\y+0.7) -- (\x+1.0,\y-6.6) -- (\x-0.5,\y-6.6) --  (\x-0.5,\y+0.0);

  \def\x{1.7}
  \def\y{8.1+1}
\draw[circle, -,dashed, very thick,rounded corners=18pt] (\x+2.8,\y+0.0)--(\x-0.5,\y+0.0)--(\x-0.5,\y+1.5)--(\x+13.0,\y+1.5) --(\x+13.0,\y) -- (\x+2.8,\y+0.0);

  \def\x{1.7}
  \def\y{5.3+1}
\draw[circle, -,dashed, very thick,rounded corners=18pt] (\x+2.8,\y+0.0)--(\x-0.5,\y+0.0)--(\x-0.5,\y+1.5)--(\x+13.0,\y+1.5) --(\x+13.0,\y) -- (\x+2.8,\y+0.0);

  \def\x{1.7}
  \def\y{3-0.7+1}
\draw[circle, -,dashed, very thick,rounded corners=18pt] (\x+2.8,\y+0.0)--(\x-0.5,\y+0.0)--(\x-0.5,\y+1.5)--(\x+13.0,\y+1.5) --(\x+13.0,\y) -- (\x+2.8,\y+0.0);

  \def\x{17.2}
  \def\y{-0.8+1}
\draw[circle, -,dashed, very thick,rounded corners=18pt] (\x+0.5,\y+0.0)--(\x-0.5,\y+0.0)--(\x-0.5,\y+1.5)--(\x+1.0,\y+1.5) --(\x+1.0,\y) -- (\x+0.0,\y+0.0);

  \def\x{-1.8}
  \def\y{-0.8+13}
\draw[circle, -,dashed, very thick,rounded corners=18pt] (\x+0.5,\y+0.0)--(\x-0.5,\y+0.0)--(\x-0.5,\y+1.5)--(\x+1.0,\y+1.5) --(\x+1.0,\y) -- (\x+0.0,\y+0.0);

    \def\x{0.5}
  \def\y{9.2+1}

  \def\x{1.5}
  \def\y{-6.0}
\draw[circle, -,dotted, very thick,rounded corners=8pt] (\x-0.5,\y-0.5)--(\x-0.5,\y-0.3) --(\x+1.5,\y-0.3) -- (\x+3.3,\y-2.0) -- (\x+14,\y-2.0) -- (\x+14,\y-4.7)--(\x+14,\y-10)-- (\x+16.8,\y-12.7) -- (\x+16.8,\y-13.2) -- (\x+16.5,\y-13.7)  -- (\x+15,\y-13.7) -- (\x+11.75,\y-11.7)-- (\x+1.5,\y-11.7)-- (\x+1.5,\y-3.7)-- (\x+0.5,\y-1.7)-- (\x-0.5,\y-1.7)  --  (\x-0.5,\y-0.5);

  \def\x{1.5}
  \def\y{-6.4}
\draw[circle, -,dashed, very thick,rounded corners=8pt] (\x+0.2,\y+2)--(\x+17.4,\y+2) --(\x+17.9,\y+1.5) -- (\x+17.9,\y-13)-- (\x+17.4,\y-13.5) -- (\x-0.3,\y-13.5) -- (\x-0.8,\y-13) -- (\x-0.8,\y+1.5) -- (\x-0.3,\y+2)--(\x+0.1,\y+2);

\end{tikzpicture}
}
\end{center}
\caption{Decomposition of $G\cong G/_{i=1}^5R_i\boxbackslash G/_{j=1}^6C_i$. The set $Z$ from the proof of Theorem~\ref{theorem_6} and the graph isomorphic to $G$ induced by $Z$ in $G/_{i=1}^5R_i\boxtimes G/_{j=1}^6C_i$ are indicated within the dotted region (except for the arc with label $i$).}
  \label{FourthDecomposition}
\end{figure}

\begin{proof}
It clearly suffices to define a mapping $\phi: V(G)\rightarrow V(G/_{i=1}^{m}R_i\boxbackslash G/_{j=1}^{n}C_j)$ and to prove that $\phi$ is an isomorphism from $G$ to $G/_{i=1}^{m}R_i\boxbackslash G/_{j=1}^{n}C_j$.

Let $\tilde{x}'_i$ be the new vertex replacing the sets $R_i$ with $v_{i,j}\in R_i$, $\tilde{x}''_j$  be the new vertex replacing the set $C_j$ with $v_{i,j}\in C_j$, when defining $G/_{i=1}^{m}R_i$ and $G/_{j=1}^{n}C_j$, respectively.
Consider the mapping $\phi: V(G)\rightarrow V(G/_{i=1}^{m}R_i\boxbackslash G/_{j=1}^{n}C_j)$ defined by  $\phi(v_{i,j})=(\tilde{x}_i,\tilde{x}_j)$ for all $v_{i,j}\in V(G)$.
Then $\phi$ is obviously a bijection if  $V(G/_{i=1}^{m}R_i\boxbackslash G/_{j=1}^{n}C_j)=Z$, where $Z$ is defined as $Z=\{(\tilde{x}_i,\tilde{x}_j)\mid \phi(v_{i,j})=(\tilde{x}_i,\tilde{x}_j), v_{i,j}\in V(G)\}$. 
We are going to show this later by arguing that all the other vertices of $G/_{i=1}^{m}R_i\Box G/_{j=1}^{n}C_j$ will disappear from $G/_{i=1}^{m}R_i\boxtimes G/_{j=1}^{n}C_j$. But first we are going to prove the following claim.

\begin{claim}\label{claim4}
The subgraph of $G/_{i=1}^{m}R_i\boxtimes G/_{j=1}^{n}C_j$ induced by $Z$ is isomorphic to $G$.
\end{claim}
\begin{proof}
We start with proving that $\tilde{x}'_i$ and $\tilde{x}'_j, i\neq j,$ implies $\tilde{x}'_i\neq \tilde{x}'_j$ and that $\tilde{x}''_i$ and $\tilde{x}''_j, i\neq j,$ implies $\tilde{x}''_i\neq \tilde{x}'_j$.
Because $R_i$ is the row with all vertices $v_{i,k}$ of $V(G)$ (by hypothesis) the vertices of $R_i$ are replaced by $\tilde{x}'_i$ and $R_j$ is the row with all vertices $v_{j,k}$ of $V(G)$ (by hypothesis) the vertices of $R_j$ are replaced by $\tilde{x}'_j$, and $R_i\cap R_j=\emptyset, i\neq j$, $\tilde{x}'_i$ and $\tilde{x}'_j, i\neq j,$ implies $\tilde{x}'_i\neq \tilde{x}'_j$. 
Likewise, Because $C_i$ is the column with all vertices $v_{i,k}$ of $V(G)$ (by hypothesis) the vertices of $C_i$ are replaced by $\tilde{x}'_i$ and $C_j$ is the column with all vertices $v_{j,k}$ of $V(G)$ (by hypothesis) the vertices of $C_j$ are replaced by $\tilde{x}''_j$, and $C_i\cap C_j=\emptyset$, $\tilde{x}''_i$ and $\tilde{x}''_j, i\neq j,$ implies $\tilde{x}''_i\neq \tilde{x}''_j$ . 
Next, because all vertices $v_{i,j}$ are replaced by $\tilde{x}'_i$ by $G/_{i=1}^{m}R_i$ and all vertices $v_{i,j}$ are replaced by $\tilde{x}''_j$ by $G/_{j=1}^{n}C_j$, it follows that $G/_{i=1}^{m}R_i\boxtimes G/_{j=1}^{n}C_j$ contains $(\tilde{x}'_i,\tilde{x}''_j)$ with as a result that er is a one-to-one correspondence between $v_{i,j}$ and $(\tilde{x}'_i,\tilde{x}''_j)$.
It follows that $\phi:V(G)\rightarrow Z$ is a bijection.
It remains to show that this bijection preserves the arcs and their labels. 
By hypothesis, the arcs of the rows of the subgraphs of $G_M$ are asynchronous with respect to the arcs of the columns of the subgraphs of $G_M$ and the arcs of the subgraphs of $G_M$ are asynchronous with respect to the arcs of the subgraphs of $G_B$.
For each arc $a$ of $G$ with $\mu(a)=(v_{i,j},v_{k,l})$, $i\neq k$, there is an arc $b$ of $G/_{i=1}^mR_i$ with $\mu(b)=(\tilde{x}'_i,\tilde{x}'_k)$ and $\lambda(a)=\lambda(b)$ and for each arc $c$ of $G$ with $\mu(c)=(v_{i,j},v_{k,l})$, $j\neq l$, there is an arc $d$ of $G/_{j=1}^nC_j$  with $\mu(d)=(\tilde{x}''_j,\tilde{x}''_l)$ and $\lambda(c)=\lambda(d)$.
Because the arcs of each subgraph $G_{B_x}$ of $G_B$ are synchronous arcs, we have that if $a$ is a synchronous arc of $G_{B_x}$ then $G/_{i=1}^{m}R_i\boxtimes G/_{j=1}^{n}C_j$ contains an arc $d$ with $\mu(d)=((\tilde{x}'_i,\tilde{x}''_j),(\tilde{x}'_k,\tilde{x}''_l))$  and $\lambda(a)=\lambda(d)$.
Because the arcs of the rows of each subgraph $G_{M_x}$ of $G_M$ are asynchronous arcs with respect to the arcs of the columns of $G_{M_x}$ (and vice versa), we have that if $a$ is such an asynchronous arc of a subgraph of $G_M$ then $G/_{i=1}^{m}R_i\boxtimes G/_{j=1}^{n}C_j$ contains arcs $d$ with $\mu(d)=((\tilde{x}'_i,\tilde{x}''_k),(\tilde{x}'_j,\tilde{x}''_k))$  and $\lambda(a)=\lambda(d)$ or arcs $d$ with $\mu(d)=((\tilde{x}'_i,\tilde{x}''_k),(\tilde{x}'_i,\tilde{x}''_l))$  and $\lambda(a)=\lambda(d)$ 
Because $G$ consists of subgraphs of $G_M$ and $G_B$ only, there are no other arcs $a$ of $G$.
Therefore, the subgraph of $G/_{i=1}^{m}R_i\boxtimes G/_{j=1}^{n}C_j$ induced by $Z$ is isomorphic to $G$.
\end{proof}
We continue with the proof of Theorem~\ref{theorem_7}. 
It remains to show that all other vertices of $G/_{i=1}^{m}R_i\boxtimes G/_{j=1}^{n}C_j$, except for the vertices of $Z$, disappear from $G/_{i=1}^{m}R_i\boxtimes G/_{j=1}^{n}C_j$.
First, we observe that all vertices of $Z$ are of the type $(\tilde{x}'_i,\tilde{x}''_j)$.
Therefore, it suffices to show that vertices of the types $(\tilde{x}'_i,v_j)$, $(v_i,\tilde{x}''_i)$ and $(v_i,v_j)$ do not exist in $G/_{i=1}^{m}R_i\boxtimes G/_{j=1}^{n}C_j$ and the vertices $(\tilde{x}'_i, \tilde{x}''_j)$ of $G/_{i=1}^{m}R_i\boxtimes G/_{j=1}^{n}C_j$ that are not in $Z$ will disappear from $G/_{i=1}^{m}R_i\boxtimes G/_{j=1}^{n}C_j$.
Because all vertices $v_{i,j}$ of $G$ are in $R_{i}$, the set of vertices $\{v_{i,j}\}$ is replaced by the vertex $\tilde{x}'_i$ and therefore $v_{i,j}$ does not exist in $G/_{i=1}^{m}R_i$ and all vertices $v_{i,j}$ of $G$ are in $C_{j}$, the set of vertices of $v_{i,j}$ is replaced by the vertex $\tilde{x}''_j$ and therefore $v_{i,j}$ does not exist in $G/_{i=1}^{n}C_i$.
Hence, by definition of the Cartesian product, vertices of the types $(\tilde{x}'_i,v_j)$, $(v_i,\tilde{x}''_i)$ and $(v_i,v_j)$ do not exist in $G/_{i=1}^{m}R_i\boxtimes G/_{j=1}^{n}C_j$.
By definition of the VRSP, if a vertex $(\tilde{x}'_i, \tilde{x}''_j)\notin Z$ has $level~0$ in $G/_{i=1}^{m}R_i\boxtimes G/_{j=1}^{n}C_j$, $(\tilde{x}'_i, \tilde{x}''_j)$ is removed from $G/_{i=1}^{m}R_i\boxtimes G/_{j=1}^{n}C_j$. 
This followes directly from $\phi$ mapping the source of $G$ into the source of the graph induced by $Z$.
Therefore, assume $(\tilde{x}'_k, \tilde{x}''_l)\notin Z$ has $level > 0$ in $G/_{i=1}^{m}R_i\boxtimes G/_{j=1}^{n}C_j$.
For a vertex  $(\tilde{x}'_k, \tilde{x}''_l)\notin Z$ to have level$>0$ in $G/_{i=1}^{m}R_i\boxtimes G/_{j=1}^{n}C_j$ there must be an arc $a$ in $G/_{i=1}^{m}R_i\boxtimes G/_{j=1}^{n}C_j$ with $\mu(a)=((\tilde{x}'_i, \tilde{x}''_j),(\tilde{x}'_k, \tilde{x}''_l))$ with either $(\tilde{x}'_i, \tilde{x}''_j)\in Z$ or $(\tilde{x}'_i, \tilde{x}''_j)\notin Z$.
In the case that $(\tilde{x}'_i, \tilde{x}''_j)\notin Z$ we can recursively backtrack the paths until we reach a vertex $(\tilde{x}'_i, \tilde{x}''_j)\in Z$ or we reach a vertex $(\tilde{x}'_i, \tilde{x}''_j)\notin Z$ with $level~0$.
In the former case, the arc $a$ cannot exist, because otherwise $a$ corresponds to an arc $b$ in $A(G_{M_x})$ or $a$ corresponds to an arc $b$ in $A(G_{B_x})$ with $\mu(b)=(v_{i,j},v_{k,l})$ and $\lambda(a)=\lambda(b)$. 
But such an arc $b$ cannot exist because for such an arc $b$ we have that there exists an arc $c$ in  $G/_{i=1}^{m}R_i$ with $\mu(c)=(\tilde{x}'_i,\tilde{x}_k')$ and $\lambda(b)=\lambda(c)$ and there exists an arc $d$ in  $G/_{j=1}^{n}C_j$ with $\mu(d)=(\tilde{x}''_j,\tilde{x_l}'')$ and $\lambda(b)=\lambda(c)=\lambda(d)$. 
Therefore, there exists an arc $e$ with $\mu(e) = ((\tilde{x}'_i,\tilde{x}''_j), (\tilde{x}'_k,\tilde{x}''_l))$ and $\lambda(e)=\lambda(a)$ in the graph induced by $Z$.
This contradicts the assumption $(\tilde{x}'_i, \tilde{x}''_j)\notin Z$.
In the latter case, the vertex $(\tilde{x}'_i, \tilde{x}''_j)$ is removed from $G/_{i=1}^{m}R_i\boxtimes G/_{j=1}^{n}C_j$ together with the arc $a$ with $\mu(a)=((\tilde{x}'_i, \tilde{x}''_j),(\tilde{x}'_{i_1}, \tilde{x}''_{j_1}))$, recursively, until the arc $a'$ with $\mu(a')=((\tilde{x}'_{i_n}, \tilde{x}''_{j_n}),(\tilde{x}'_{k}, \tilde{x}''_{l}))$ is removed.
This completes the proof of Theorem~\ref{theorem_7}. 
\end{proof}
\section{Future work}
In this paper, we believe that we have supplied all ingredients with which we can decompose a labelled acyclic directed multigraph with respect to the VRSP.
Based on Theorems~\ref{theorem_1},~\ref{theorem_2},~\ref{theorem_5},~\ref{theorem_6}~and~\ref{theorem_7}  we believe that graphs that cannot be decomposed by any of these theorems must be a prime-graph with respect to the VRSP. 
But, this still has to be proved in future work.

\end{document}